\documentclass[12pt]{amsart}
\usepackage[margin=1.4in]{geometry}

\usepackage[style=alphabetic, backend=bibtex, maxbibnames=99]{biblatex}
\AtEveryBibitem{\ifentrytype{article}{\clearfield{url} \clearfield{issn}}{}}

\usepackage[colorlinks=true]{hyperref}
\hypersetup{citecolor=Sepia, linkcolor=Sepia, urlcolor=MidnightBlue}

\usepackage{lmodern}

\usepackage{amsmath}        
\usepackage{amsfonts}       
\usepackage{amsthm}         
\usepackage{bbding}         
\usepackage{bm}             
\usepackage{graphicx}       
\usepackage{fancyvrb}       
\usepackage[usenames, dvipsnames]{xcolor}  
\usepackage[colorinlistoftodos, textsize=tiny]{todonotes}

\usepackage{amssymb}%
\usepackage{mathtools}
\usepackage[shortlabels]{enumitem}
	\setlist[enumerate,1]{label=(\roman*), font=\normalfont} 
\usepackage{phaistos}
\usepackage{fontawesome}

\theoremstyle{plain}
\newtheorem{theorem}{Theorem}[section]
\newtheorem{lemma}[theorem]{Lemma}
\newtheorem{proposition}[theorem]{Proposition}
\newtheorem{corollary}[theorem]{Corollary}

\numberwithin{equation}{section}

\theoremstyle{definition}
\newtheorem{definition}[theorem]{Definition}
\newtheorem{remark}[theorem]{Remark}
\newtheorem{example}[theorem]{Example}

\newcounter{CountQuestions}

\newtheorem{question}[CountQuestions]{Question}

\newcommand{\goto}{\rightarrow}

\newcommand{\uv}[1]{``{#1}"} 

\DeclareMathOperator{\ndeg}{ndeg}

\DeclareMathOperator{\spn}{span}

\newcommand{\Z}{\mathbb{Z}}

\newcommand{\mB}{\mathcal{B}}
\newcommand{\mBc}[1]{\mathcal{B}_{\mathrm{#1}}}

\newcommand{\boldzeta}{\boldsymbol{\zeta}}
\newcommand{\boldxi}{\boldsymbol{\xi}}

\newcommand{\ve}{\varepsilon}

\newcommand{\simv}{\stackrel{{\scriptsize{{v}}}}{\sim}}
\newcommand{\simvw}{\stackrel{{v}_0}{\sim}}
\newcommand{\simsim}{\stackrel{{\scriptsize{\mathrm{sim}}}}{\sim}}

\newcommand{\iw}[1]{\mathfrak{i}_{\mathrm{W}}(#1)} 
\newcommand{\iql}[1]{\mathfrak{i}_{\mathrm{d}}(#1)} 

\newcommand{\an}{\mathrm{an}}

\newcommand{\fullsplitpat}[1]{\mathrm{fSP}(#1)}
\newcommand{\fsp}[1]{\mathrm{fSP}(#1)}

\newcommand{\ort}{\:\bot\:}

\newcommand{\sqf}[1]{\langle #1 \rangle} 

\newcommand{\pf}[1]{\langle\!\langle #1 \rangle\!\rangle} 
\newcommand{\nf}{\hat{\nu}} 
\newcommand{\simf}{\hat{\sigma}} 

\begin{document}
\title[Vishik equivalence and similarity of $p$-forms]{Vishik equivalence and similarity of quasilinear $p$-forms and totally singular quadratic forms}
\author{Krist\'yna Zemkov\'a}
\address{Fakult\"at f\"ur Mathematik, Technische Universit\"at Dortmund, D-44221 Dortmund,
Germany}
\address{Department of Mathematics and Statistics, University of Victoria, Victoria BC V8W 2Y2, Canada}
\email{zemk.kr@gmail.com}%
\date{\today}
\date{\today}
\begin{abstract}
For quadratic forms over fields of characteristic different from two, there is a so-called Vishik criterion, giving a purely algebraic characterization of when two quadratic forms are motivically equivalent. In analogy to that, we define Vishik equivalence on quasiliner $p$-forms. We study the question whether Vishik equivalent $p$-forms must be similiar. We prove that this is not true for quasilinear $p$-forms in general, but we find some families of totally singular quadratic forms (i.e., of quasilinear $2$-forms) for which the question has a positive answer.
\end{abstract}
\thanks{This work was supported by DFG project HO 4784/2-1. The author further acknowledges support from the Pacific Institute for the Mathematical Sciences and a partial support from the National Science and Engineering Research Council of Canada. The research and findings may not reflect those of these institutions.\\ \indent The author reports there are no competing interests to declare.}
\maketitle

\section{Introduction}

Vishik defined an equivalence relation on quadratic forms over fields of characteristic other than $2$ (see \cite{Vish98-prep}, and also \cite{Kar00,Vish04-notes}): $\varphi$ and $\psi$ are equivalent if and only if 
\begin{equation}\label{Eq:VishikCrit}
\dim\varphi=\dim\psi \quad \text{and} \quad \iw{\varphi_E}=\iw{\psi_E}\ \text{ for any field extension } E/F;
\end{equation}
Vishik proved that this equivalence coincides with motivic equivalence. Hence, \eqref{Eq:VishikCrit} is nowadays known as \emph{Vishik's criterion} for motivic equivalence. Then the question was raised whether the equivalence defined by Vishik also coincides with similarity; in other words, the following question was asked:

\begin{question} \label{Not2Q}
Are quadratic forms satisfying Vishik's criterion \eqref{Eq:VishikCrit} necessarily similar?
\end{question}

Izhboldin proved in \cite{Izh98} and \cite{Izh00} that the answer to Question~\ref{Not2Q} is positive for odd-dimensional quadratic forms, but negative for even dimensional forms of dimension greater or equal to $8$ (except possibly for the dimension $12$). In \cite{Hof15-Mot}, Hoffmann proves that, under some conditions on the base field, Question~\ref{Not2Q} has positive answer for even-dimensional quadratic forms as well.

\bigskip

In the case of fields of characteristic $2$, we have to distinguish between different types of quadratic forms -- nonsingular and totally singular (the two extreme cases) and singular (the mixed type). In \cite{Hof04}, totally singular quadratic forms over fields of characteristic $2$ have been generalized to \emph{quasilinear $p$-forms} over fields of characteristic $p$. 

Let $F$ be a field of characteristic $p$, and let $\varphi$ and $\psi$ be quasilinear $p$-forms over $F$. Inspired by Vishik's criterion, we define \emph{Vishik equivalence} of $\varphi$ and $\psi$ as
\begin{equation*}
\dim\varphi=\dim\psi \qquad \text{and} \qquad \iql{\varphi_E}=\iql{\psi_E}\ \text{ for any field extension } E/F,
\end{equation*}
where $\iql{\tau}$ is the \emph{defect} (sometimes also called the \emph{quasilinear index}) of the quasilinear $p$-form $\tau$. In analogy to Question~\ref{Not2Q}, we ask:

\setcounter{CountQuestions}{16}
\begin{question} \label{MQ}
Are Vishik equivalent quasilinear $p$-forms necessarily similar?
\end{question}

We show in examples~\ref{Ex:CounterExA_p-forms} and~\ref{Ex:CounterExB_p-forms} that Question~\ref{MQ} has a negative answer for $p$-forms if $p>3$. Therefore, in Section~\ref{Sec:Vishik_TSQF}, we will focus on totally singular quadratic forms, i.e., on the case when $p=2$. In this case, we can give a positive answer to Question~\ref{MQ} at least for some families of forms. For example, we prove the following:

\begin{theorem} [{cf. Theorem~\ref{Th:TSQF_Minimal}}]
Let $\varphi, \psi$ be totally singular quadratic forms over $F$ such that $\varphi_{\an}$ is minimal over $F$. If $\varphi$ and $\psi$ are Vishik equivalent, then they are similar.
\end{theorem}

\noindent See Theorem~\ref{Th:TSQF_SPNsummary} and Corollary~\ref{Cor:VishikTSQFtotal} for more families for which the answer to Question~\ref{MQ} is positive. We would also like to point out Theorem~\ref{Thm:SimilarityFactorsVishik_pforms} which shows that Vishik equivalence preserves similarity factors, even in the case of quasilinear $p$-forms.

\bigskip

This paper is based on the second and third chapter of the author's PhD thesis \cite{KZdis}.


\section{Preliminaries}\label{Sec:Prel}

All fields in this article are of characteristic $p>0$. Whenever we talk about (totally singular) quadratic forms, we assume $p=2$.


\subsection{Quasilinear $p$-forms} 

Most of the necessary background on quasilinear $p$-forms can be found in \cite{Hof04}; we include it here for the readers' convenience.

\begin{definition}
Let $F$ be a field and $V$ a finite-dimensional $F$-vector space. A \emph{quasilinear $p$-form} (or simply a \emph{$p$-form}) \emph{over $F$} is a map ${\varphi:V\goto F}$ with the following properties:
\begin{enumerate}[(1)]
	\item $\varphi(av)=a^p\varphi(v)$ for any $a\in F$ and $v\in V$, \label{Enum:Defpform1}
	\item $\varphi(v+w)=\varphi(v)+\varphi(w)$ for any $v,w\in V$. \label{Enum:Defpform2}
\end{enumerate}
The dimension of $\varphi$ is defined as $\dim\varphi=\dim V$. 
\end{definition}

Any $p$-form $\varphi$ on an $F$-vector space $V$ can be associated with the polynomial $\sum_{i=1}^na_iX_i^p\in F[X_1,\dots,X_n]$, where $a_i=\varphi(v_i)$ with $\{v_1,\dots,v_n\}$ a basis of $V$. In such case, we write $\varphi$ as $\sqf{a_1,\dots,a_n}$.

We have $c\sqf{a_1,\dots,a_n}=\sqf{ca_1,\dots,ca_n}$ for any $c\in F^*$. If $\sqf{a_1,\dots,a_n}$ and $\sqf{b_1\dots,b_m}$ are two $p$-forms over $F$, then we define
\begin{align*}
\sqf{a_1,\dots,a_n}\ort\sqf{b_1\dots,b_m}&=\sqf{a_1,\dots,a_n,b_1\dots,b_m},\\
\sqf{a_1,\dots,a_n}\otimes\sqf{b_1\dots,b_m}&=a_1\sqf{b_1\dots,b_m}\ort\dots\ort a_n\sqf{b_1\dots,b_m}.
\end{align*}
Moreover, for a positive integer $k$, we write $k\times\varphi$ for the $p$-form $\varphi\ort\dots\ort\varphi$ consisting of $k$ copies of $\varphi$.

\bigskip

Two $p$-forms $\varphi:V\goto F$ and $\psi:W\goto F$ are called \emph{isometric} (denoted $\varphi\simeq\psi$) if there exists a bijective homomorphism $f:V\goto W$ of vector spaces such that $\varphi(v)=\psi(f(v))$ for any $v\in V$. If $f$ is not bijective but injective, then $\varphi$ is called a \emph{subform} of $\psi$ (denoted $\varphi\subseteq\psi$); then there exists a $p$-form $\sigma$ over $F$ such that $\psi\simeq\varphi\ort\sigma$. If $\varphi\simeq c\psi$ for some $c\in F^*$, then $\varphi$ and $\psi$ are called \emph{similar} (denoted $\varphi\simsim\psi$).

A $p$-form  $\varphi:V\goto F$ is called \emph{isotropic} if $\varphi(v)=0$ for some $v\in V\setminus\{0\}$; otherwise, $\varphi$ is called \emph{anisotropic}. The $p$-form $\varphi$ can be written as ${\varphi\simeq\sigma\ort k\times\sqf{0}}$ with $\sigma$ an anisotropic $p$-form over $F$ and $k$ a non-negative integer. Then $\sigma$ is unique up to isometry, and it is called the \emph{anisotropic part of $\varphi$} (denoted $\varphi_{\an}$). The integer $k$ is called the \emph{defect} of $\varphi$ (denoted $\iql{\varphi}$). 

Let $\varphi$ be a $p$-form on an $F$-vector space $V$; we set 
\begin{align*}
&D_F(\varphi)=\{\varphi(v)~|~v\in V\} \quad &\text{and}& \quad &D_F^*(\varphi)=D_F(\varphi)\setminus\{0\},\\
 &G_F^*(\varphi)=\{x\in F^*~|~x\varphi\simeq\varphi\} \quad &\text{and}& \quad &G_F(\varphi)=G_F^*(\varphi)\cup\{0\}.
\end{align*}
Note that $D_F(\varphi)$ is an $F^p$-vector space; in particular, if $\varphi\simeq\sqf{a_1,\dots,a_n}$, then $D_F(\varphi)=\spn_{F^p}\{a_1,\dots,a_n\}$. 

\begin{lemma}[{\cite[Prop.~2.6]{Hof04}}] \label{Lemma:PickAnisp-subform} \label{Lem:p-subform} 
Let $\varphi$ be a $p$-form over $F$.
\begin{enumerate}
	\item Let $\{c_1,\dots,c_k\}$ be any $F^p$-basis of the vector space $D_F(\varphi)$. Then we have $\varphi_{\an}\simeq\sqf{c_1,\dots,c_k}$.
	\item If $a_1,\dots,a_m\in D_F(\varphi)$, then $\sqf{a_1,\dots,a_m}_{\an}\subseteq\varphi$.
\end{enumerate}
\end{lemma}

It follows that, for any $a,b\in F$ and $x\in F^*$, we have
\[\sqf{a,b}\simeq\sqf{a+b,b}\quad \text{and} \quad \sqf{a}\simeq\sqf{ax^p}.\]

\bigskip

For a $p$-form $\varphi$ on an $F$-vector space $V$ and a field extension $E/F$, we denote by $\varphi_E$ the $p$-form on the $E$-vector space $V_E=E\otimes V$ defined by $\varphi_E(e\otimes v)=e^p\varphi(v)$ for any $e\in E$ and $v\in V$. It was proved in \cite[Lemma~5.1]{Hof04} that if $\varphi\simeq\sqf{a_1,\dots,a_n}$, then $(\varphi_E)_\an\simeq\sqf{a_{i_1},\dots,a_{i_k}}$ for some $\{i_1,\dots,i_k\}\subseteq\{1,\dots,n\}$.

\subsection{Quasi-Pfister forms and quasi-Pfister neighbors} 

For $n>0$, an \emph{$n$-fold quasi-Pfister form} is a $p$-form $\pf{a_1,\dots,a_n}=\pf{a_1}\otimes\cdots\otimes\pf{a_n}$, where $\pf{a}=\sqf{1,a,\dots,a^{p-1}}$. Moreover, $\sqf{1}$ is called the \emph{$0$-fold quasi-Pfister form}. 

A $p$-form $\varphi$ over $F$ is called a \emph{quasi-Pfister neighbor} if $c\varphi\subseteq\pi$ for some $c\in F^*$ and a quasi-Pfister form $\pi$ over $F$ such that $\dim\varphi>\frac{1}{p}\dim\pi$. The value $\dim\pi-\dim\varphi$ is called the \emph{codimension} of $\varphi$.

\bigskip

In the following lemma, we summarize some of the most important properties of quasi-Pfister forms and quasi-Pfister neighbors.

\begin{lemma}[{\cite[Prop.~4.6]{Hof04}, \cite[Lemma~2.6]{Scu13}, \cite[Prop.~4.14]{Hof04}}] \label{Lemma:PFandPN}
Let $\pi\simeq\pf{a_1,\dots,a_n}$ be a quasi-Pfister form over $F$, $\varphi$ a quasi-Pfister neighbor of $\pi$, and $E/F$ a field extension. The followings hold:
\begin{enumerate}
	\item $G_E(\pi)=D_E(\pi)$.
	\item There either exist $k>0$ and a subset $\{i_1,\dots,i_k\}\subseteq\{1,\dots,n\}$ such that $(\pi_E)_\an\simeq\pf{a_{i_1},\dots,a_{i_k}}_E$, or $(\pi_E)_\an\simeq\sqf{1}_E$.
	\item $\varphi_E$ is isotropic if and only if $\pi_E$ is isotropic.
\end{enumerate}
\end{lemma}

\subsection{$p$-bases, norm fields and norm forms} \label{Subsec:p-bases}

Let $a_1,\dots,a_n\in F$. We call the set $\{a_1,\dots,a_n\}$ \emph{$p$-independent over $F$} if $[F^p(a_1,\dots,a_n):F^p]=p^n$, and \emph{$p$-dependent over $F$} otherwise. The set $\{a_1,\dots,a_n\}$ is $p$-independent over $F$ if and only if the quasi-Pfister form $\pf{a_1,\dots,a_n}$ is anisotropic over $F$ (\cite[Cor.~2.8]{KZ23-fsp}). A (possibly infinite) set $\mathcal{S}\subseteq F$ is called \emph{$p$-independent over $F$} if any finite subset of $\mathcal{S}$ is $p$-independent over $F$. 

Let $E$ be a field such that $F^p\subseteq E\subseteq F$; a set $\mB\subseteq E$ is called a \emph{$p$-basis of $E$ over $F$} if $\mB$ is $p$-independent over $F$ and $E=F^p(\mB)$. Such a $p$-basis always exists, and any subset of $E$ that is $p$-independent over $F$ can be extended into a $p$-basis of $E$ over $F$ (\cite[Cor.~A.8.9]{CSAandGC}). A set $\mB=\{b_i~|~i\in I\}$ is a $p$-basis of $E$ over $F$ if and only if
\[\widehat{\mB}=\left\{\prod_{i\in I}b_i^{\lambda(i)}~\Big|~\lambda: I\goto\{0,\dots,p-1\}, \lambda(i)=0 \text{ for almost all } i\in I\right\}\]
is an $F^p$-linear basis of $E$ (\!\!\cite[pg.~27]{Pick50}); in such case, any $a\in E$ can be expressed \emph{uniquely} as
\[a=\sum_\lambda x_\lambda^p\prod_{i\in I}b_i^{\lambda(i)}\]
for some $x_\lambda\in F$, almost all of them zero. By abuse of language, we say that \emph{$a$ can be expressed uniquely with respect to $\mB$}.

\bigskip

Let $\varphi$ be a $p$-form over $F$. We define the \emph{norm field of $\varphi$ over $F$} as the field 
\[N_F(\varphi)=F^p\left(\frac{a}{b}~\Bigl|~a,b\in D_F^*(\varphi)\right).\] 
By the definition, we have $N_F(\varphi)=N_F(\varphi_\an)$ and also that $N_F(\varphi)=N_F(c\varphi)$ for any $c\in F^*$. Furthermore, if $\psi$ is another $p$-form over $F$ satisfying $\psi\simeq\varphi$, then $N_F(\psi)=N_F(\varphi)$. Finally, if $\tau_\an\subseteq c\varphi_\an$ for some $p$-form $\tau$ over $F$ and $c\in F^*$, then $N_F(\tau)\subseteq N_F(\varphi)$.

\begin{lemma}[{\cite[Lemma~4.2 and Cor.~4.3]{Hof04}}]\label{Lemma:NormFieldBasics}
Let $\varphi$ be a $p$-form over $F$.
\begin{enumerate}
	\item If $\varphi\simeq\sqf{a_0,\dots,a_n}$ for some $n\geq1$ and $a_i\in F$, $0\leq i\leq n$, with $a_0\neq0$, then $N_F(\varphi)=F^p\bigl(\frac{a_1}{a_0},\dots,\frac{a_n}{a_0}\bigr)$.
	\item Suppose $N_F(\varphi)=F^p(b_1,\dots,b_m)$ for some $b_i\in F^*$, $1\leq i \leq m$,  and let $E/F$ be a field extension. Then $N_E(\varphi)=E^p(b_1,\dots,b_m)$.
\end{enumerate}
\end{lemma}
The following proposition provides an alternative possibility for a determination of the norm field.

\begin{proposition} \label{Prop:p-forms_NormfieldCharacterisation}
Let $\varphi=\sqf{1,a_1,\dots, a_n}$ be a $p$-form over $F$, and let $b_1,\dots,b_m\in F$. Then
\[
\bigl(\varphi_{F(\sqrt[p]{b_1}, \dots, \sqrt[p]{b_m})}\bigr)_{\an}\simeq\sqf{1} 
\ \ \ \Longleftrightarrow \ \ \ 
N_F(\varphi)\subseteq F^p(b_1,\dots,b_m).
\]
If, moreover, $\{b_1,\dots,b_m\}\subseteq\{a_1,\dots,a_n\}$, then
\[
\bigl(\varphi_{F(\sqrt[p]{b_1}, \dots, \sqrt[p]{b_m})}\bigr)_{\an}\simeq\sqf{1} 
\ \ \ \Longleftrightarrow \ \ \ 
N_F(\varphi)= F^p(b_1,\dots,b_m).
\] 
In particular, if $m$ is minimal with this property, then $\ndeg_F\varphi=p^m$.
\end{proposition}

\begin{proof}
We set $E=F\bigl(\sqrt[p]{b_1}, \dots, \sqrt[p]{b_m}\bigr)$. 

First, $(\varphi_E)_{\an}\simeq\sqf{1}$ is equivalent to $\spn_{E^p}\{1\}=\spn_{E^p}\{1,a_1,\dots,a_n\}$, which holds if and only if $a_i\in E^p$ for all $1\leq i \leq n$. As $E^p=F^p(b_1,\dots,b_m)$, the latter condition is equivalent to $F^p(a_1,\dots,a_n)\subseteq F^p(b_1,\dots,b_m)$. But $F^p(a_1,\dots,a_n)=N_F(\varphi)$ by Lemma~\ref{Lemma:NormFieldBasics}, so we are done.

If $\{b_1,\dots,b_m\}\subseteq\{a_1,\dots,a_n\}$, then $F^p(b_1,\dots,b_m)\subseteq N_F(\varphi)$, and the claim follows by the previous case.
\end{proof}

Note that $N_F(\varphi)$ is a finite field extension of $F^p$, and hence there always exists a finite $p$-basis $\{b_1,\dots,b_n\}$ of $N_F(\varphi)$ over $F$, i.e., $N_F(\varphi)=F^p(b_1,\dots,b_n)$ with $[F^p(b_1,\dots,b_n):F^p]=p^n$. Then $p^n$ is called the \emph{norm degree of $\varphi$ over $F$} and denoted by $\ndeg_F{\varphi}$. There is a relation between the norm degree and the dimension of a $p$-form.

\begin{lemma}[{\cite[Prop.~4.8]{Hof04}}]\label{Lemma:PrereqForNormForm} 
Let $\varphi$ be a nonzero $p$-form with $\ndeg_F\varphi=p^n$. Then $n+1\leq\dim\varphi_\an\leq p^n$.
\end{lemma}

By \cite[Cor.~A.8.9]{CSAandGC}, any $p$-generating set contains a $p$-basis. Therefore, part (i) of Lemma~\ref{Lemma:NormFieldBasics} implies the following:

\begin{lemma}\label{Lem:BasisSubform}
Let $\varphi\simeq\sqf{a_0,\dots,a_n}$ for some $n\geq1$ and $a_0,\dots,a_n\in F$ with $a_0\neq0$. Moreover, suppose that $\ndeg_F\varphi=p^k$. Then there exists a subset $\{i_1,\dots,i_k\}\subseteq\{1,\dots,n\}$ such that $\left\{\frac{a_{i_1}}{a_0},\dots,\frac{a_{i_k}}{a_0}\right\}$ is a $p$-basis of $N_F(\varphi)$ over $F$.

 In particular, if $1\in D_F^*(\varphi)$, then there exist $b_1,\dots,b_n\in F$ such that $\varphi\simeq\sqf{1,b_1,\dots,b_n}$ and $N_F(\varphi)=F^p(b_1,\dots,b_k)$.
\end{lemma}

\bigskip

Let $\varphi$ be a $p$-form over $F$ with $\ndeg_F{\varphi}=p^n$ and ${N_F(\varphi)=F^p(a_1,\dots,a_n)}$ (so, in particular, $\{a_1,\dots,a_n\}$ is $p$-independent over $F$). Then we define the \emph{norm form} of $\varphi$ over $F$, denoted by $\nf_F(\varphi)$, as the (necessarily anisotropic) quasi-Pfister form $\pf{a_1,\dots,a_n}$. It follows that $\nf_F(\varphi)$ is the smallest quasi-Pfister form that contains a scalar multiple of $\varphi_\an$ as its subform. In particular, if $\varphi_\an$ is a quasi-Pfister form itself, then we have $\nf_F(\varphi)\simeq\varphi_\an$.

\subsection{Isotropy} 

Let $\varphi$ be an anisotropic $p$-form over $F$, and let $E/F$ be a field extension. If $E/F$ is purely transcendental or separable, then $\varphi_E$ remains anisotropic (\!\!\cite[Prop.~5.3]{Hof04}).

\begin{lemma}[{\cite[Lemma~2.27]{Scu16-Hoff}}] \label{Lemma:p-forms_BasicIsotropy}
Let $\varphi$ be a $p$-form over $F$ and let $a\in F\setminus F^p$. Then:
\begin{enumerate}
	\item $D_{F(\sqrt[p]{a})}(\varphi)=D_F(\pf{a}\otimes\varphi)=\sum_{i=0}^{p-1}a^iD_F(\varphi)$. \label{Lemma:p-forms_BasicIsotropy_i}
	\item $\iql{\varphi_{F(\sqrt[p]{a})}}=\frac1p\iql{\pf{a}\otimes\varphi}$.
	\item \label{Lemma:p-forms_BasicIsotropy_Part_NormField} 
					If $\varphi$ is anisotropic and $\varphi_{F(\sqrt[p]{a})}$ is isotropic, then $a\in N_F(\varphi)$.
	\item \label{Item:DimOverInseparabelFE} 
					$\dim(\varphi_{F(\sqrt[p]{a})})_{\an}\geq\frac1p\dim\varphi_{\an}$.
	\item Equality holds in \ref{Item:DimOverInseparabelFE} if and only if there exists a $p$-form $\gamma$ over $F$ such that $\varphi_{\an}\simeq\pf{a}\otimes\gamma$. 
\end{enumerate}
\end{lemma}

We want to point out that $N_F(\varphi)$ is the smallest field extension of $F^p$ with property \ref{Lemma:p-forms_BasicIsotropy_Part_NormField} of Lemma~\ref{Lemma:p-forms_BasicIsotropy}. More precisely, we have the following:

\begin{proposition}\label{Prop:IsotropyAndNormField}
Let $\varphi$ be an anisotropic $p$-form over $F$ and $E/F^p$ be a field extension such that the following holds for any $a\in F^*$:
\begin{equation} \label{Eq:NecCondIsotropypRoot}
\iql{\varphi_{F(\sqrt[p]{a})}}>0 \quad \Longrightarrow \quad a\in E. 
\tag{{\smaller[4]{\PHrosette}}}
\end{equation}
Then $N_F(\varphi)\subseteq E$.
\end{proposition}

\begin{proof}
Let $\ndeg_F\varphi=p^k$. Since neither the isotropy nor the norm field depends on the choice of the representative of the similarity class, we can assume $1\in D_F^*(\varphi)$. Invoking Lemma~\ref{Lem:BasisSubform}, we find $b_1,\dots,b_n\in F$ with $n\geq k$ such that $\varphi\simeq\sqf{1,b_1,\dots,b_n}$ and $\{b_1,\dots,b_k\}$ is a $p$-basis of $N_F(\varphi)$ over $F$. Since $\varphi_{F(\sqrt[p]{b_i})}$ is obviously isotropic for each $1\leq i\leq k$, we get by the assumption \eqref{Eq:NecCondIsotropypRoot} that $b_1,\dots,b_k\in E$. Since $E$ is a field containing $F^p$, it follows that $N_F(\varphi)=F^p(b_1,\dots,b_k)\subseteq E$.
\end{proof}

\begin{lemma}\label{Lem:quasiPFaddingSlot} 
Let $\varphi$ be an anisotropic quasi-Pfister form or an anisotropic quasi-Pfister neighbor over $F$, and let $x\in F^*$. Then $\varphi\otimes\pf{x}$ is isotropic if and only if $x\in N_F(\varphi)$.
\end{lemma}

\begin{proof}
First, we prove the lemma in the case when $\varphi\simeq\pi$ is a \mbox{quasi-Pfister} form. Assume that $\pi\simeq\pf{a_1,\dots,a_n}$, i.e., $N_F(\pi)=F^p(a_1,\dots,a_n)$. Note that $\pi$ is anisotropic. Then the form ${\pi\otimes\pf{x}}$ is isotropic if and only if $(\pi\otimes\pf{x})_{\an}\simeq\pi$ if and only if $F^p(a_1,\dots,a_n)=F^p(a_1,\dots,a_n,x)$ if and only if $x\in N_F(\pi)$.

Now suppose that $\varphi$ is a quasi-Pfister neighbor. Then $c\varphi\subseteq\pi$ for $\pi\simeq\nf_F(\varphi)$ and some $c\in F^*$. It follows that $c\varphi\otimes\pf{x}$ is a quasi-Pfister neighbor of $\pi\otimes\pf{x}$, and hence $c\varphi\otimes\pf{x}$ is isotropic if and only if $\pi\otimes\pf{x}$ is isotropic (Lemma~\ref{Lemma:PFandPN}). By the first part of the proof, the latter is equivalent to $x\in N_F(\pi)$. Since $N_F(\pi)=N_F(c\varphi)=N_F(\varphi)$, the claim follows.
\end{proof}

\subsection{Minimal forms}

In this paper, we will work with minimal forms; they are, in a way, a counterpart to quasi-Pfister forms.

\begin{definition}
Let $\varphi$ be an anisotropic $p$-form over $F$. We call $\varphi$ \emph{minimal over $F$} if $\ndeg_F\varphi=p^{\dim\varphi-1}$.
\end{definition}

Note that any one-dimensional anisotropic $p$-form is minimal. Minimal $p$-forms of dimension at least two are described by the following lemma:

\begin{lemma}[{\cite[Lemma~2.16]{KZ23-fsp}}]\label{Lem:PropertiesOfMinimalForms_pforms}
Let $\varphi,\psi$ be $p$-forms over $F$ of dimension at least two.
\begin{enumerate}
	\item The $p$-form $\sqf{1,a_1,\dots,a_n}$ is minimal over $F$ if and only if the set $\{a_1,\dots,a_n\}$ is $p$-independent over $F$.
	\item Minimality is not invariant under field extensions.
	\item If $\varphi$ is minimal over $F$ and $\psi\simsim\varphi$, then $\psi$ is minimal over $F$.
	\item If $\varphi$ is minimal over $F$ and $\psi\subseteq c\varphi$ for some $c\in F^*$, then $\psi$ is minimal over $F$.
	\item If $\ndeg_F\varphi=p^k$ with $k\geq1$, then $\varphi$ contains a minimal subform of dimension $k+1$.
\end{enumerate}
\end{lemma}

We can characterize elements represented by a given minimal $p$-form.

\begin{lemma} \label{Lemma:pforms_Minimal_Repr}
Let $\varphi$ be a minimal $p$-form over $F$ with $1\in D_F(\varphi)$. Let $n\geq2$ and $\{b_1,\dots,b_n\}\subseteq D_F(\varphi)$ be $p$-independent over $F$. Denote 
\[S=\bigl\{\lambda: \{1,\dots,n\}\rightarrow\{0,\dots,p-1\}\bigr\},\vspace{-1mm}\]
and consider 
\[\beta=\sum_{\lambda\in S}x_\lambda^p\prod_{i=1}^nb_i^{\lambda(i)}\] 
with $x_\lambda\in F$ for all $\lambda\in S$. Then the following are equivalent:\vspace{1mm}

\noindent\begin{minipage}{\textwidth}
\begin{enumerate}
	\item $\beta \in D_F(\varphi)$, \label{Lemma:pforms_Minimal_Repr_i}
	\item $x_\lambda=0$ for all $\lambda\in S$ such that $\sum_{i=1}^n\lambda(i)\geq2$, i.e., $\beta=y_0^p+\sum_{i=1}^nb_iy_i^p$ for some $y_i\in F$, $0\leq i \leq n$. \label{Lemma:pforms_Minimal_Repr_ii}
\end{enumerate}
\end{minipage}
\end{lemma}

\begin{proof}
Write $\tau\simeq\sqf{1,b_1,\dots,b_n}$. Then \ref{Lemma:pforms_Minimal_Repr_ii} holds if and only if $\beta\in D_F(\tau)$. Since $\tau\subseteq\varphi$ by the assumptions, the implication \ref{Lemma:pforms_Minimal_Repr_ii}~$\Rightarrow$~\ref{Lemma:pforms_Minimal_Repr_i} is obvious. 

To prove \ref{Lemma:pforms_Minimal_Repr_i}~$\Rightarrow$~\ref{Lemma:pforms_Minimal_Repr_ii}, suppose that \ref{Lemma:pforms_Minimal_Repr_ii} does not hold (e.g., $\beta=b_1b_2$). Then we can see that $\beta\notin D_F(\tau)$. Hence, the $p$-form $\tau\ort\sqf{\beta}$ is anisotropic. On the other hand, $\beta\in F^p(b_1,\dots,b_n)$, so $\tau\ort\sqf{\beta}$ is not minimal over $F$ by part (i) of Lemma \ref{Lem:PropertiesOfMinimalForms_pforms}. If $\beta\in D_F(\varphi)$, then $\tau\ort\sqf{\beta}\subseteq\varphi$, in which case $\varphi$ could not be minimal over $F$ by part (iv) of the same lemma, which is a contradiction; therefore, $\beta\notin D_F(\varphi)$.
\end{proof}

\subsection{Vishik equivalence}

In this subsection, we introduce Vishik equivalence of $p$-forms. The name originates from \emph{Vishik criterion} that, in characteristic different from $2$, gives an algebraic characterization of the motivic equivalence (see \cite{Hof15-Mot}).

\begin{definition}\label{Def:Vishik}
Let $\varphi, \psi$ be $p$-forms over $F$. We say that they are \emph{Vishik equivalent} and write $\varphi\simv\psi$ if $\dim\varphi=\dim\psi$ and the following holds:
\begin{equation}\label{Eq:DefVishik}
\iql{\varphi_E}=\iql{\psi_E} \quad \forall\;E/F. \tag{$v$}
\end{equation}
\end{definition}

\begin{remark}
For quadratic forms $\varphi$ and $\psi$ over $F$ (not necessarily totally singular) such that $\dim\varphi=\dim\psi$, we define Vishik equivalence as follows (see \cite[Def.~1.57]{KZdis}):
\[
\varphi\simv\psi \qquad \stackrel{\text{def}}{\Longleftrightarrow} \qquad \iw{\varphi_E}=\iw{\psi_E} \quad \& \quad \iql{\varphi_E}=\iql{\psi_E} \quad \forall\;E/F.
\]
Note that this definition agree with Definition~\ref{Def:Vishik} in the special case of totally singular quadratic forms.
\end{remark}

\begin{lemma} \label{Lem:VishikWrtExt} \label{Lemma:VishikAnisotropy}
Let $\varphi$, $\psi$ be $p$-forms over $F$ such that $\varphi\simv\psi$, and let $E/F$ be a~field extension. Then $(\varphi_E)_{\an}\simv(\psi_E)_{\an}$. In particular, $\varphi_{\an}\simv\psi_{\an}$.
\end{lemma}

\begin{proof}
Denote $\varphi'=(\varphi_E)_{\an}$ and $\psi'=(\psi_E)_{\an}$. Let $K/E$ be another field extension. It holds that $\iql{\varphi_K}=\iql{\psi_K}$ and $\iql{\varphi_E}=\iql{\psi_E}$. Thus, 
\[\iql{\varphi'_K}=\iql{\varphi_K}-\iql{\varphi_E}=\iql{\psi_K}-\iql{\psi_E}=\iql{\psi'_K}.\] 
Therefore, $\varphi'\simv\psi'$.
\end{proof}

Since similar $p$-forms must be of the same dimension and have the same defects over any field, we get:

\begin{lemma}\label{Lemma:SimImpliesVishik}
If $\varphi\simsim\psi$, then $\varphi\simv\psi$.
\end{lemma}

Among others, Vishik equivalence implies that a $p$-form is isotropic over the function field of any Vishik equivalent $p$-form. 

\begin{lemma} \label{Lemma:VishikToIsotropy}
Let $\varphi$, $\psi$ be $p$-forms over $F$ of dimension at least two. If $\varphi\simv\psi$, then $\varphi_{F(\psi)}$ and $\psi_{F(\varphi)}$ are isotropic.
\end{lemma}

\begin{proof}
Since $\varphi_{F(\varphi)}$ is isotropic and we know that $\iql{\psi_{F(\varphi)}}=\iql{\varphi_{F(\varphi)}}$, we get that $\psi_{F(\varphi)}$ is isotropic. Symmetrically, $\varphi_{F(\psi)}$ is isotropic.
\end{proof}

\begin{lemma} \label{Lemma:VishikToNF}
Let $\varphi$, $\psi$ be anisotropic $p$-forms over $F$ such that $\varphi\simv\psi$. Then $\nf_F(\varphi)\simeq\nf_F(\psi)$, $N_F(\varphi)=N_F(\psi)$ and $\ndeg_F(\varphi)=\ndeg_F(\psi)$.
\end{lemma}

\begin{proof}
The claim is trivial if $\dim\varphi=\dim\psi=1$. Assume that the dimension is at least two. By Lemma~\ref{Lemma:VishikToIsotropy}, we know that $\varphi_{F(\psi)}$ and $\psi_{F(\varphi)}$ are isotropic. Thus, this is a consequence of \cite[Lemma 7.12]{Hof04}.
\end{proof}

\section{Vishik equivalence on $p$-forms} \label{Sec:Vishik_p-forms}

As an easy consequence of some know results, we get that Question~\ref{MQ} has a positive answer for quasi-Pfister forms and quasi-Pfister neighbors of codimension one.

\begin{proposition} \label{Prop:Vishik_PF}
Let $\varphi$, $\psi$ be anisotropic $p$-forms such that $\varphi\simv\psi$. If $\varphi$ is a quasi-Pfister form or a quasi-Pfister neighbor of codimension one, then $\varphi\simsim\psi$. 
\end{proposition}

\begin{proof}
By Lemma~\ref{Lemma:VishikToNF}, there exists a quasi-Pfister form $\pi$ over $F$ such that $\pi\simeq\nf_F(\varphi)\simeq\nf_F(\psi)$.

If $\varphi$ is a quasi-Pfister form, then we can find $c,d\in F^*$ such that $c\varphi\simeq\pi\supseteq d\psi$; thus, $\varphi\simsim\psi$ for dimension reasons. 

Now assume that $\varphi$ is a quasi-Pfister neighbor of  $\pi$ of codimension $1$. Then we have $\dim\psi=\dim\varphi=\dim\pi-1$. Moreover, $c\psi\subseteq\nf_F(\psi)\simeq\pi$ for some ${c\in F^*}$. Therefore, $\psi$ is a quasi-Pfister neighbor of $\pi$ of codimension $1$. Using \cite[Prop.~4.15]{Hof04}, any two such $p$-forms are similar, i.e., $\varphi\simsim\psi$ as claimed.
\end{proof}

\subsection{Weak Vishik equivalence}

Proving that two $p$-forms are Vishik equivalent might be difficult. Moreover, we do not always need the full strength of the Vishik equivalence; therefore, we define a weaker version.

\begin{definition}
Let $\varphi$, $\psi$ be $p$-forms over $F$. We say that they are \emph{weakly Vishik equivalent} and write $\varphi\simvw\psi$ if $\dim\varphi=\dim\psi$ and the following holds:
\begin{equation}
\iql{\varphi_{F(\sqrt[p]{a})}}=\iql{\psi_{F(\sqrt[p]{a})}} \qquad \forall\;a\in F. \tag{$v_0$}
\end{equation}
\end{definition}

The weak Vishik equivalence has three (rather obvious) properties.

\begin{lemma} \label{Lemma:WeakVishikBasicProperties}
Let $\varphi$, $\psi$ be $p$-forms over $F$. Then:
\begin{enumerate}
	\item If $\varphi\simvw\psi$, then  $\iql{\varphi}=\iql{\psi}$.
	\item $\varphi\simvw\psi$ if and only if $\varphi_{\an}\simvw\psi_{\an}$.
	\item If $\varphi\simv\psi$, then $\varphi\simvw\psi$.
\end{enumerate}
\end{lemma}

\begin{proof}
Part (i) follows directly from the definition, because we include the equality of the defects over the field $F(\sqrt[p]{1})\simeq F$. 
Part (ii) is then an easy consequence of part (i). Finally, part (iii) is trivial.
\end{proof}

To prove that Vishik equivalent forms have the same norm field, we used the isotropy over the function fields of each other (see Lemma~\ref{Lemma:VishikToNF}). Another possibility to determine the norm field is to use Proposition~\ref{Prop:p-forms_NormfieldCharacterisation}, which only needs one particular purely inseparable field extension of exponent one. But for weakly Vishik equivalent forms, neither of these is at our disposal; nevertheless, thanks to Proposition~\ref{Prop:IsotropyAndNormField}, we are still able to prove that they have the same norm field. 

\begin{proposition} \label{Prop:WeakVishikSameNormField}
Let $\varphi$, $\psi$ be $p$-forms over $F$ such that $\varphi\simvw\psi$. Then $\nf_F(\varphi)\simeq\nf_F(\psi)$, $N_F(\varphi)=N_F(\psi)$ and $\ndeg_F(\varphi)=\ndeg_F(\psi)$.
\end{proposition}

\begin{proof}
We prove only $N_F(\varphi)=N_F(\psi)$, then the rest follows. Invoking Lemma~\ref{Lemma:WeakVishikBasicProperties} and recalling that the norm field takes into account only the anisotropic part of a $p$-form, we can assume that $\varphi$ and $\psi$ are anisotropic. By the definition of the weak Vishik equivalence and by part \ref{Lemma:p-forms_BasicIsotropy_Part_NormField} of Lemma~\ref{Lemma:p-forms_BasicIsotropy}, we have for any $a\in F$:
\[\iql{\varphi_{F(\sqrt[p]{a})}}>0 \quad \Longrightarrow \quad \iql{\psi_{F(\sqrt[p]{a})}}>0 \quad \Longrightarrow \quad a\in N_F(\psi).\]
Thus, we can apply Proposition~\ref{Prop:IsotropyAndNormField} on the $p$-form $\varphi$ and the field $E=N_F(\psi)$; we obtain $N_F(\varphi)\subseteq N_F(\psi)$. By the symmetry of the argument, we get $N_F(\varphi)=N_F(\psi)$.
\end{proof}

Since (weakly) Vishik equivalent forms are always of the same dimension, we immediately get:

\begin{corollary} \label{Cor:WeakVishikAndMinimality}
Let $\varphi$, $\psi$ be $p$-forms over $F$ such that $\varphi\simvw\psi$. 
\begin{enumerate}
	\item If $\varphi$ is minimal over $F$, then $\psi$ is also minimal over $F$.
	\item If $\varphi$ is a quasi-Pfister form or a quasi-Pfister neighbor of codimension one over $F$, then $\varphi\simsim\psi$.
\end{enumerate}
\end{corollary}

\begin{proof}
By Proposition~\ref{Prop:WeakVishikSameNormField}, we have $\ndeg_F\varphi=\ndeg_F\psi$; since $\dim\varphi=\dim\psi$, part (i) follows immediately. To prove part (ii), note that we also have $\nf_F(\varphi)\simeq\nf_F(\psi)$, and so we can proceed as in the proof of Proposition~\ref{Prop:Vishik_PF}.
\end{proof}

\subsection{Similarity factors}

In this subsection, we will show that weak Vishik equivalence (and hence Vishik equivalence as well) preserves divisibility by \mbox{quasi-Pfister} forms. In particular, we will prove that if $\varphi\simvw\psi$ for some $p$-forms $\varphi$, $\psi$ with $\varphi$ divisible by a quasi-Pfister form $\pi$, then $\psi$ is divisible by $\pi$, too.

\bigskip

Let $\varphi$ be a $p$-form defined over $F$. Recall that we write
\[G_F(\varphi)=\{x\in F^*~|~x\varphi\simeq\varphi\}\cup\{0\};\]
we call the nonzero elements of this set the \emph{similarity factors} of $\varphi$. As observed in \cite[Prop. 6.4]{Hof04}, the set $G_F(\varphi)$ together with the usual operations is a finite field extension of $F^p$ inside $N_F(\varphi)$; in particular, there exists a $p$-independent set $\{a_1,\dots,a_m\}\subseteq F^*$ such that $G_F(\varphi)=F^p(a_1,\dots,a_m)$. We denote $\simf_F(\varphi)\simeq\pf{a_1,\dots,a_m}$ and call it the \emph{similarity form} of $\varphi$ over $F$. Moreover, again by \cite[Prop.~6.4]{Hof04}, there exists a $p$-form $\gamma$ over $F$ such that $\varphi_{\an}\simeq\simf_F(\varphi)\otimes\gamma$.  It holds that $D_F(\simf_F(\varphi))=G_F(\varphi)$.

We will show that weak Vishik equivalent $p$-forms have the same similarity factors. But first, we need a simple lemma.

\begin{lemma} \label{Lem:PFdivisibility}
Let $\pi_1$, $\pi_2$ be anisotropic quasi-Pfister forms. Then $\pi_1\subseteq\pi_2$ if and only if $\pi_2\simeq\pi_1\otimes\gamma$ for some $p$-form $\gamma$ over $F$. In that case, $\gamma$ can be chosen to be a quasi-Pfister form.
\end{lemma}

\begin{proof}
Write $\pi_1\simeq\pf{a_1,\dots,a_r}$ and $\pi_2\simeq\pf{b_1,\dots,b_s}$, which means that we have $N_F(\pi_1)=F^p(a_1,\dots,a_r)$ and $N_F(\pi_2)=F^p(b_1,\dots,b_s)$. By \cite[Prop.~4.6]{Hof04}, there is a bijection between finite field extensions of $F^p$ inside $F$ and Pfister $p$-forms over $F$, which implies that $\pi_1\subseteq\pi_2$ if and only if $N_F(\pi_1)\subseteq N_F(\pi_2)$. 

If this holds, then we can extend $\{a_1,\dots,a_r\}$ to a {$p$-basis} of $F^p(b_1,\dots,b_s)$; thus, $F^p(b_1,\dots,b_s)=F^p(a_1,\dots,a_r,a_{r+1},\dots,a_s)$ for some $a_{r+1},\dots,a_s\in F^*$, and so $\pf{b_1,\dots,b_s}\simeq\pf{a_1,\dots,a_r}\otimes\pf{a_{r+1},\dots,a_s}$.

On the other hand, assume $\pi_2\simeq\pi_1\otimes\gamma$ for a $p$-form $\gamma$. If $1\in D_F^*(\gamma)$, then we can write $\gamma\simeq\sqf{1}\ort\gamma'$ for a suitable $p$-form $\gamma'$; in that case, $\pi_2\simeq\pi_1\ort\pi_1\otimes\gamma'$ and we are done. More generally, if $c\in D_F^*(\gamma)$, then $c\in D_F^*(\pi_2)=G_F^*(\pi_2)$. Now $\pi_2\simeq c\pi_2\simeq\pi_1\otimes c\gamma$ with $1\in D_F^*(c\gamma)$, so we are done by the previous case.
\end{proof}

\begin{theorem} \label{Thm:SimilarityFactorsVishik_pforms}
Let $\varphi$, $\psi$ be anisotropic $p$-forms over $F$ such that $\varphi\simvw\psi$. Then $G_F(\varphi)=G_F(\psi)$.
\end{theorem}

\begin{proof}
The inclusion $F^p\subseteq G_F(\psi)$ is obvious. Hence, pick $a\in G_F(\varphi)\setminus F^{p}$. Since $G_F(\varphi)=D_F(\simf_F(\varphi))$ is a field, we get that $a^2,\dots,a^{p-1}\in D_F(\simf_F(\varphi))$ and $1\in D_F(\simf_F(\varphi))$. As the $p$-form $\sqf{1,a,\dots,a^{p-1}}\simeq\pf{a}$ is anisotropic, we get $\pf{a}\subseteq\simf_F(\varphi)$ by Lemma~\ref{Lem:p-subform}; therefore, $\pf{a}$ divides $\simf_F(\varphi)$ by Lemma~\ref{Lem:PFdivisibility}. Since further $\simf_F(\varphi)$ divides $\varphi$, we can find a $p$-form $\varphi'$ defined over $F$ such that $\varphi\simeq\pf{a}\otimes\varphi'$. 

Set $E=F(\sqrt[p]{a})$; it holds that $(\varphi_E)_{\an}\simeq (\varphi'_E)_{\an}$. Therefore, we have $\dim(\varphi'_E)_{\an}\leq\frac{1}{p}\dim\varphi$. But the reverse inequality is true by Lemma~\ref{Lemma:p-forms_BasicIsotropy}; hence, $\varphi'_E$ is anisotropic and $\dim(\varphi_E)_{\an}=\frac{1}{p}\dim\varphi$. 

Since $\varphi\simvw\psi$, we have $\dim(\psi_E)_{\an}=\frac{1}{p}\dim\psi$. Let $\psi'$ be a $p$-form over $F$ such that $\psi'_E\simeq(\psi_E)_{\an}$. As it holds that
\[D_F(\psi)\subseteq D_E(\psi)=D_E(\psi')=D_F(\pf{a}\otimes\psi'),\]
we get $\psi\subseteq\pf{a}\otimes\psi'$ by Lemma~\ref{Lem:p-subform}. Comparing the dimensions implies that $\psi\simeq\pf{a}\otimes\psi'$, and hence $a\in G_F(\psi)$. 

All in all, we have proved $G_F(\varphi)\subseteq G_F(\psi)$. The other inclusion follows by the symmetry of the argument. 
\end{proof}

\begin{lemma}  \label{Lem:MultiplesIsometry_pforms}
Let $\tau$ be a $p$-form defined over $F$ and $K/F$ be a field extension such that $K^p\subseteq G_F(\tau)$. Let $\varphi$, $\psi$ be $p$-forms defined over $F$, anisotropic over $K$ and such that $\varphi_K\simeq\psi_K$. Then $\varphi\otimes\tau\simeq\psi\otimes\tau$ over $F$.
\end{lemma}

\begin{proof}
Let $\varphi\simeq\sqf{b_1,\dots,b_m}$ and $\psi\simeq\sqf{c_1,\dots,c_m}$. It follows from the assumptions that $\{b_1,\dots,b_m\}$ and $\{c_1,\dots,c_m\}$ are two bases of the same $K^p$-vector space. Recall that we can get one basis from the other by a finite series of operations of the following type: an exchange of two basis elements, scalar multiplication of one basis element, and adding one basis element to another basis element. Thus, it is sufficient to show that the following hold for any $0\leq i,j\leq m$, $i\neq j$, and $a\in K^p$:
\begin{enumerate}
	\item $\sqf{b_0,\dots, b_i, \dots, b_j, \dots,b_m}\otimes\tau \simeq \sqf{b_0,\dots, b_j, \dots, b_i, \dots,b_m}\otimes\tau$,
	\item $\sqf{b_0,\dots,b_i,\dots,b_m}\otimes\tau \simeq \sqf{b_0,\dots,ab_i,\dots,b_m}\otimes\tau$,
	\item $\sqf{b_0,\dots,b_i,\dots, b_j,\dots,b_m}\otimes\tau \simeq \sqf{b_0,\dots,b_i+b_j,\dots, b_j,\dots,b_m}\otimes\tau$.
\end{enumerate}
Part (i) is obvious. Part (ii) follows from the fact that $b_i\tau\simeq ab_i\tau$ since $a\in G_F(\tau)$. To prove part (iii), note that $b_i\tau\ort b_j\tau\simeq (b_i+b_j)\tau\ort b_j\tau$.
\end{proof}

\begin{proposition} \label{Prop:DividingByPF_pforms}
Let $\pi\simeq\pf{a_1,\dots,a_n}$ be an anisotropic quasi-Pfister form over $F$ and $K=F(\sqrt[p]{a_1},\dots,\sqrt[p]{a_n})$. Assume that $\varphi\simeq\pi\otimes\varphi'$ and $\psi\simeq\pi\otimes\psi'$ are anisotropic $p$-forms over $F$ such that $\varphi'_K\simsim\psi'_K$. Then $\varphi\simsim\psi$.
\end{proposition}

\begin{proof}
Without loss of generality, assume that $1\in D_F(\varphi')\cap D_F(\psi')$. Write $\varphi'\simeq\sqf{1,b_1,\dots,b_m}$, and let $c\in K^*$ be such that $c\psi'_K\simeq\varphi'_K$. Then 
\[c\in D_K(\varphi')=\spn_{K^p}\{1,b_1,\dots,b_m\}\subseteq F;\] 
thus, $c\psi'$ is defined over $F$. By \cite[Thm.~3.3]{KZ23-fsp}, we have $p^n\iql{\varphi'_K}=\iql{\varphi}=0$, and hence $\varphi'_K$ is anisotropic; analogously, we get that $c\psi'_K$ is anisotropic. Now Lemma~\ref{Lem:MultiplesIsometry_pforms} implies that $\pi\otimes\varphi'\simeq\pi\otimes c\psi'$. As $\pi\otimes c\psi'\simeq c(\pi\otimes\psi')$, the claim follows.
\end{proof}

\begin{remark}
Theorem~\ref{Thm:SimilarityFactorsVishik_pforms} and Proposition~\ref{Prop:DividingByPF_pforms} can be used for a simplification of the problem, whether Vishik equivalence implies the similarity. Namely, we can restrict ourselves to the case of forms  with the similarity factors $F^p$: Let $\varphi$, $\psi$ be two anisotropic $p$-forms over $F$ such that $\varphi\simv\psi$. Since ${G_F(\varphi)=G_F(\psi)}$ by Theorem~\ref{Thm:SimilarityFactorsVishik_pforms}, we can write $\varphi\simeq\pi\otimes\varphi'$ and $\psi\simeq\pi\otimes\psi'$ with the quasi-Pfister form $\pi\simeq\simf_F(\varphi)\simeq\simf_F(\psi)$ and some anisotropic $p$-forms $\varphi'$, $\psi'$ defined over $F$. Let $\pi\simeq\pf{a_1,\dots,a_n}$, and put $K=F(\sqrt[p]{a_1},\dots, \sqrt[p]{a_n})$. Then $\varphi'_K$, $\psi'_K$ are anisotropic by \cite[Prop.~5.7]{Hof04}. Hence $(\varphi_K)_{\an}\simeq\varphi'_K$ and $(\psi_K)_{\an}\simeq\psi'_K$, and we get $\varphi'_K\simv\psi'_K$ by Lemma~\ref{Lem:VishikWrtExt} (note that to apply this lemma, we need the full strength of the Vishik equivalence, the weak Vishik equivalence does not have to be sufficient here). If we knew that $\varphi'_K\simsim\psi'_K$, it would follow by Proposition~\ref{Prop:DividingByPF_pforms} that $\varphi\simsim\psi$.
\end{remark}

\subsection{Counterexamples}

Recall that by Proposition~\ref{Prop:Vishik_PF}, Question~\ref{MQ} has a positive answer for quasi-Pfister forms and quasi-Pfister neighbors of codimension one (i.e., Vishik equivalence is sufficient for the $p$-forms to be similar). It would be nice if Question~\ref{MQ} had a positive answer for all $p$-forms. Unfortunately, this is not true in general -- as we will see, it is not even true for all quasi-Pfister neighbors. We will provide two examples of pairs of $p$-forms that are Vishik equivalent but not similar. However, both of the counterexamples have two things in common: First, the considered $p$-forms are subforms of $\pf{a}$, and so they have norm degree one. Second, they do not work for $p=2$ and $p=3$. We will focus on the characteristic two case in Section~\ref{Sec:Vishik_TSQF}, but the case $p=3$ remains open as well as $p$-forms of higher norm degrees.

\bigskip

Before we can get to the counterexamples, we need some lemmas.

\begin{lemma} \label{Lem:2dimIsotropy}
Let $1\leq k,l\leq p-1$, $a\in F\setminus F^p$, and let $E/F$ be a field extension. Then the following are equivalent:

\noindent\begin{minipage}{\textwidth}
\begin{enumerate}
	\item $\sqf{1,a^k}$ is isotropic over $E$,
	\item $\sqf{1,a^l}$ is isotropic over $E$,
	\item $a^k\in E^p$,
	\item $a^l\in E^p$.
\end{enumerate}
\end{minipage}
\end{lemma}

\begin{proof}
Assume that $\sqf{1,a^k}$ is isotropic over $E$. This is equivalent to the existence of $x,y\in E$, at least one (and hence both) of them nonzero, such that $x^p+a^ky^p=0$, which is equivalent to $a^k=\frac{(-x)^p}{y^p}\in E^p$.  This proves both (i) $\Leftrightarrow$ (iii) and (ii) $\Leftrightarrow$ (iv).

Now let $1\leq t\leq p-1$ be such that $kt\equiv l\pmod p$. If $a^k\in E^p$, then $a^{kt}\in E^p$, and hence also $a^l\in E^p$. This proves (iii) $\Rightarrow$ (iv) and, by the symmetry of the argument, also (iv) $\Rightarrow$ (iii).
\end{proof}

\begin{lemma} \label{Lem:2dimSimilarity}
Let $a\in F\setminus F^p$ and $k,l\in\Z$ with $k,l\not\equiv0\pmod p$. Then $\sqf{1,a^k}\simsim\sqf{1,a^l}$ if and only if $k\equiv\pm l\pmod p$.
\end{lemma}

\begin{proof}
If $k\equiv l\pmod p$, then $k=pm+l$ for some $m\in\Z$, and so we have $\sqf{1,a^k}\simeq\sqf{1, a^{pm+l}}\simeq\sqf{1,a^l}$. If $k\equiv-l\pmod p$, then we can write $k=pn-l$ for some $n\in\Z$, and then we have 
\[\sqf{1,a^k}=\sqf{1,a^{pn-l}}\simeq\sqf{1,a^{-l}}\simsim a^l\sqf{1,a^{-l}}\simeq\sqf{a^l,1}\simeq\sqf{1, a^l}.\]

To prove the opposite direction, assume that $\sqf{1,a^k}\simsim\sqf{1,a^l}$, i.e., there exists $c\in F^*$ such that $c\sqf{1,a^k}\simeq\sqf{1,a^l}$. Then necessarily $c\in D_F(\sqf{1,a^l})$, so we can find $x,y\in F$ such that $c=x^p+a^ly^p$. On the other hand, $1\in D_F(c\sqf{1,a^k})$, and thus $1=cu^p+ca^kv^p$ for some $u,v\in F$. Putting these together, we have 
\[1=(x^p+a^ly^p)u^p+(x^p+a^ly^p)a^ku^p=(xu)^p+a^l(yu)^p+a^k(xv)^p+a^{k+l}(yv)^p.\] 
Suppose $k\not\equiv\pm l\pmod p$. Then $\tau=\sqf{1, a^k, a^l, a^{k+l}}$ is a subform of $\pf{a}$, and hence anisotropic. Thus, $\tau$ represents every element of $D_F(\tau)$ uniquely, so we get $xu\neq 0$ and $yu=xv=yv=0$. This implies $x,u\neq0$ and $y=v=0$; hence, $c=x^p$ and $\sqf{1,a^k}\simeq\sqf{1,a^l}$. Thus, $a^l\in D_F(\sqf{1,a^k})$, which is impossible for $l\not\equiv0,k\pmod p$.
\end{proof}

\begin{lemma} \label{Lemma:ViskikEquivQuasiPN}
Let $a\in F\setminus F^p$ and $\varphi$, $\psi$ be quasi-Pfister neighbors of $\pf{a}$ of the same dimension. Then $\varphi\simv\psi$.
\end{lemma}

\begin{proof}
Obviously, both $\varphi$ and $\psi$ are anisotropic. Let $E/F$ be any field extension; it holds $(\varphi_E)_{\an},(\psi_E)_{\an}\subseteq (\pf{a}_E)_{\an}$. Then either $\pf{a}_E$ is anisotropic, and hence both $\varphi_E$ and $\psi_E$ are anisotropic, or we have $(\pf{a}_E)_{\an}\simeq\sqf{1}_E$ by Lemma~\ref{Lemma:PFandPN}, in which case $(\varphi_E)_{\an}\simeq\sqf{1}_E\simeq(\psi_E)_{\an}$. Consequently, $\varphi\simv\psi$.
\end{proof}

\begin{example}\label{Ex:CounterExA_p-forms}
Let $p>3$ and $\varphi\simeq\sqf{1,a}$, $\psi\simeq\sqf{1,a^2}$ be $p$-forms with $a\in F\setminus F^p$. Note that $\varphi$ and $\psi$  are quasi-Pfister neighbors of $\pf{a}$; therefore, $\varphi\simv\psi$ by Lemma~\ref{Lemma:ViskikEquivQuasiPN}. On the other hand, Lemma \ref{Lem:2dimSimilarity} ensures that $\varphi$ and $\psi$ are not similar for any $p>3$.

Note that the problems with characteristics two and three are different. If $p=2$, then $\varphi$ is anisotropic while $\psi$ is isotropic. If $p=3$, it holds that $a^2\varphi\simeq\psi$.
\end{example}

\begin{example}\label{Ex:CounterExB_p-forms}
Let $p=5$ and $a\in F\setminus F^5$. Set $\pi\simeq\sqf{1,a,a^2,a^3,a^4}\simeq\pf{a}$, $\varphi\simeq\sqf{1,a,a^2}$ and $\psi\simeq\sqf{1,a,a^3}$. It follows by Lemma~\ref{Lemma:ViskikEquivQuasiPN}  that $\varphi\simv\psi$.

Assume that $c\in F$ is such that $c\psi\simeq\varphi$. Note that 
\begin{align*} 
D_F(\varphi)&=\{x_0^5+ax_1^5+a^2x_2^5~|~x_i\in F, \ 0\leq i\leq2\}, \\
D_F(\pi)&=\{x_0^5+ax_1^5+a^2x_2^5+a^3x_3^5+a^4x_4^5~|~x_i\in F, \ 0\leq i\leq4\},
\end{align*}
and the expression of any element of $D_F(\varphi)$ (resp. $D_F(\pi)$) is unique thanks to the anisotropy of $\varphi$ (resp. $\pi$). In particular, $\varphi$ does not represent any term of the form $a^3x^5$ or $a^4x^5$ with $x\in F^*$.

As $c\in D_F(\varphi)$, we can write $c=x^5+ay^5+a^2z^5$ for some $x,y,z\in F$. Since $ca=ax^5+a^2y^5+a^3z^5\in D_F(\varphi)$, it follows that $z=0$. Furthermore, $ca^3=a^3x^5+a^4y^5+(az)^5\in D_F(\varphi)$ implies that $x=y=0$. But then $c=0$, which is absurd. Therefore, $\varphi$ and $\psi$ are not similar.
\end{example}

\section{Vishik equivalence on totally singular quadratic forms} \label{Sec:Vishik_TSQF}

In this section, we will answer Question~\ref{MQ} for some families of totally singular quadratic forms. Note that totally singular quadratic forms are the special case of $p$-forms with $p=2$, so we will use some results from the previous section. In particular, by Proposition~\ref{Prop:Vishik_PF}, we already know that Question~\ref{MQ} has a positive answer in case of quasi-Pfister forms and quasi-Pfister neighbors of codimension one. 

There is one tool which is not available for $p>2$ (at least not this strong version of it, cf. the note behind Question~4.7 in \cite{Scu13}) that we will use repeatedly:  

\begin{proposition}[{\cite[Prop.~2.11]{Lag06}}] \label{Prop:tsqf_MainTool} \label{Prop:IsotropyOverQuadrExt}
Let $a\in F\setminus F^2$ and let $\varphi$ be an anisotropic totally singular quadratic form over $F$ with $\dim\varphi\geq2$. If $\varphi$ becomes isotropic over $E=F(\sqrt{a})$, then there exists a totally singular quadratic form $\tau$ over $F$ such that $\dim\tau=\iql{\varphi_E}$ and $\tau\otimes\sqf{1,a}\subseteq\varphi$.
\end{proposition}

\subsection{Minimal quadratic forms}

Recall that a minimal quadratic form is an anisotropic totally singular quadratic form $\varphi$ over $F$ such that $\ndeg_F\varphi=2^{\dim\varphi-1}$. 

We start with some preparatory lemmas on forms of dimensions two and three. Note that all anisotropic totally singular quadratic forms of dimensions two and three are minimal by Lemma~\ref{Lemma:PrereqForNormForm}. 

\begin{lemma} \label{Lem:2dimCong}
Let $b,c\in F\setminus F^{2}$. Then $c\in D_F(\sqf{1,b})$ if and only if ${b\in D_F(\sqf{1,c})}$ if and only if $\sqf{1,b}\simeq\sqf{1,c}$.
\end{lemma}

\begin{proof}
$c\in D_F(\sqf{1,b})$ holds if we can find $x,y\in F$ such that $x^2+by^2=c$; as $c\notin F^{2}$, we have $y\neq 0$. It follows $\sqf{1,b}\simeq\sqf{1,by^2}\simeq\sqf{1,x^2+by^2}\simeq\sqf{1,c}$, and thus also $b\in D_F(\sqf{1,c})$.
\end{proof}

\begin{lemma}\label{Lem:3dimSim}
Let $a,b, x,y\in F$ and $y\neq0$. Then  
\[\sqf{1, a,bx^2+aby^2}\simeq\left(a+\left(\tfrac{x}{y}\right)^2\right)\sqf{1, a,b}.\]
\end{lemma}

\begin{proof}
We have 
\begin{multline*}
\sqf{1,a,bx^2+aby^2}
\simeq \left\langle{1,a+\left(\tfrac{x}{y}\right)^2, b\left(a+\left(\tfrac{x}{y}\right)^2\right)}\right\rangle \\
\simeq \left(a+\left(\tfrac{x}{y}\right)^2\right) \left\langle{ a+\left(\tfrac{x}{y}\right)^2,1, b}\right\rangle
\simeq \left(a+\left(\tfrac{x}{y}\right)^2\right) \sqf{1, a, b}.
\qedhere
\end{multline*}
\end{proof}

\begin{lemma}\label{Lem:3dimRewr}
Let $a,b,c\in F\setminus F^{2}$, and suppose that $a\in D_F(\sqf{1,c,bc})$. Then there exists $s\in F$ such that $\sqf{1,c,bc}\simeq(a+s^2)\sqf{1,a,b}$.
\end{lemma}

\begin{proof}
Let $x,y,z\in F$ be such that
\begin{equation} \label{Eq:3dimRewr_aDef}
a=x^2+cy^2+bcz^2.
\end{equation}

Suppose first that $z=0$; then $a\in D_F(\sqf{1,c})$, and we have $\sqf{1,c}\simeq\sqf{1,a}$ by Lemma~\ref{Lem:2dimCong}. Moreover, since $a\notin F^2$, it must be $y\neq0$, and \eqref{Eq:3dimRewr_aDef} can be rewritten as $c=\left(\frac{x}{y}\right)^2+a\left(\frac{1}{y}\right)^2$. Putting these together, we obtain 
\[
\sqf{1,c,bc}
\simeq\left\langle1,a,b\bigl(\left(\tfrac{x}{y}\right)^2+a\left(\tfrac{1}{y}\right)^2\bigr)\right\rangle
\simeq\sqf{1,a,b(x^2+a)};
\] 
hence, $\sqf{1,c,bc}\simeq\left(a+x^2\right)\sqf{1,a,b}$ by Lemma \ref{Lem:3dimSim}.

Now let $z\neq0$; then 
\[\sqf{1,c,bc}\simeq\sqf{1,c,x^2+cy^2+bcz^2}\simeq\sqf{1,a,c}.\]
Moreover, note that $z\neq0$ and $b\notin F^2$ imply $y^2+bz^2\neq0$; hence, we can rewrite \eqref{Eq:3dimRewr_aDef} as
\[ c(y^2+bz^2)^2= (a+x^2)(y^2+bz^2)=(xy)^2+ay^2+b(xz)^2+abz^2.\]
Thus, we have
\[
\sqf{1,a,c}
\simeq\sqf{1,a,b(xz)^2+abz^2}
\simeq (a+x^2)\sqf{1,a,b},
\]
where the latter isometry follows from Lemma \ref{Lem:3dimSim}. Therefore, we get $\sqf{1,c,bc}\simeq(a+x^2)\sqf{1,a,b}$ in this case, too.
\end{proof}

The following lemma mimics the situation we end up with after applying Proposition~\ref{Prop:tsqf_MainTool}.

\begin{lemma} \label{Lem:2dimSubform}
Let $\psi$ be a totally singular quadratic form over $F$ such that $1\in D_F(\psi)$. Let $a,c\in F^*$ be such that $c\sqf{1,a}$ is anisotropic over $F$, and suppose $c\sqf{1,a}\subseteq\psi$. Then 
\begin{enumerate}
	\item either $\sqf{1,c,ac}\subseteq\psi$ (this occurs if and only if $1\notin D_F(c\sqf{1,a})$),
	\item or $\sqf{1,a}\subseteq\psi$ and $c\in D_F(\sqf{1,a})$.
\end{enumerate}
\end{lemma}

\begin{proof}
If $1\notin D_F(c\sqf{1,a})$, then $\sqf{1,a,ca}$ is anisotropic, and so $\sqf{1,c,ca}\subseteq\psi$ by Lemma~\ref{Lem:p-subform}. On the other hand, if $1\in D_F(c\sqf{1,a})$, then it follows that $c\in D_F(\sqf{1,a})=G_F(\sqf{1,a})$; thus, $\sqf{1,a}\simeq c\sqf{1,a}\subseteq\psi$.
\end{proof}

\bigskip

Now we prove that (weakly) Vishik equivalent minimal quadratic forms are always similar. The proof is based mainly on the $2$-independence of the coefficients as explored in Lemma~\ref{Lemma:pforms_Minimal_Repr}.

\begin{theorem} \label{Th:TSQF_Minimal}
Let $\varphi, \psi$ be totally singular quadratic forms over $F$ such that $\varphi_{\an}$ is minimal over $F$. If $\varphi\simvw\psi$, then $\varphi\simsim\psi$.
\end{theorem}

\begin{proof}
Invoking Lemma~\ref{Lemma:WeakVishikBasicProperties} (and as obviously $\varphi\simsim\psi$ if and only if $\varphi_\an\simsim\psi_\an$), we can assume that the forms $\varphi$ and $\psi$ are anisotropic. By Lemma~\ref{Lem:PropertiesOfMinimalForms_pforms}, we can suppose that $\varphi\simeq\sqf{1,a_1,\dots,a_n}$ for some $\{a_1,\dots,a_n\}\subseteq F$ that is $2$-independent over $F$; it follows that $N_F(\varphi)=F^2(a_1,\dots,a_n)$, and \[\mBc{(0)}=\{a_1,\dots,a_n\}\] is a $2$-basis of $N_F(\varphi)$ over~$F$. Moreover, $N_F(\varphi)=N_F(\psi)$ by Proposition~\ref{Prop:WeakVishikSameNormField}, and $\psi$ is minimal over~$F$ by Corollary~\ref{Cor:WeakVishikAndMinimality}.

\smallskip

We start with an observation:
\begin{equation}\label{Eq:MainArgue}
\parbox{0.85\textwidth}{\indent \textit{For any $a\in D_F(\varphi)\setminus F^2$, the form $\varphi_{F(\sqrt{a})}$ is isotropic. As~${\varphi\simvw\psi}$, $\psi$ must be isotropic over $F(\sqrt{a})$ as well. By Proposition~\ref{Prop:tsqf_MainTool}, we can find $c\in F^*$ such that $c\sqf{1,a}\subseteq\psi$.}}
\tag{{\smaller[2]{\faEye}}}
\end{equation}
The main idea is to look at such $ca$ for some $a\in D_F(\varphi)$, and express it with respect to an appropriate $2$-basis of $N_F(\psi)$ over $F$. Applying Lemma~\ref{Lemma:pforms_Minimal_Repr}, we get that almost all coefficients must be zero. Via a combination of different values of $a$, we usually end up with the conclusion that $c$ must be a square, which means that $a\in D_F(\psi)$. In particular, we will prove that there is a scalar multiple of $\psi$ which represents all the values $1, a_1,\dots, a_n$.  

\smallskip

We divide the proof into several steps and cases. To simplify the notation and omit multiple indices, we use the same letters repeatedly -- the meaning of $s_i,u_i,x_i,y_i,z_i$ changes in each subcase (although they are usually used in similar situations). On the other hand, the meaning of $a_i,c_i,d_i, e_i$ is \uv{global}, i.e., does not change during the proof. Moreover, we consider only $2$-independence and $2$-bases over $F$, and so we omit to repeat \uv{over $F$} each time. We also would like to recall the notation from Subsection~\ref{Subsec:p-bases}: By a \uv{unique expression with respect to a $2$-basis} we actually mean the unique expression with respect to the corresponding \mbox{$F^2$-linear} basis.

\bigskip

(1) First of all, note that \eqref{Eq:MainArgue} proves the claim completely if $\dim\psi=2$. Therefore, suppose $\dim\psi\geq3$. 

As the first step, we will prove that $\psi$ contains a subform similar to $\sqf{1,a_1,a_2}$: By \eqref{Eq:MainArgue}, we find $c_1\in F^*$ such that $c_1\sqf{1,a_1}\subseteq\psi$. Since we are interested in $\psi$ only up to similarity, we can assume without loss of generality that $c_1=1$. Now, $c_2\sqf{1,a_2}\subseteq\psi$ for some $c_2\in F^*$. By Lemma~\ref{Lem:2dimSubform}, there are two possibilities: Either $\sqf{1,a_2}\subseteq\psi$ or $\sqf{1,c_2,c_2a_2}\subseteq\psi$. 

\smallskip

(1A) If $\sqf{1,a_2}\subseteq\psi$, we get $\sqf{1,a_1,a_2}\subseteq\psi$ immediately (by Lemma~\ref{Lem:p-subform}). 

\smallskip

(1B) Suppose that $\sqf{1,c_2,c_2a_2}\subseteq\psi$ holds. We have to further distinguish two cases, depending on whether $a_1$ is represented by the form $\sqf{1,c_2,c_2a_2}$ or not.

\smallskip

(1Bi)  Assume that $a_1\in D_F(\sqf{1,c_2,c_2a_2})$. Then $\sqf{1,c_2,c_2a_2}\simsim\sqf{1,a_1,a_2}$ by Lemma~\ref{Lem:3dimRewr}.

\smallskip

(1Bii) Let $a_1\notin D_F(\sqf{1,c_2,c_2a_2})$; then $\sqf{1,a_1,c_2,c_2a_2}$ is an anisotropic subform of $\psi$, and hence $\psi\simeq\sqf{1,a_1,c_2,c_2a_2,s_4,\dots,s_n}$ for some suitable $s_4,\dots, s_n\in F^*$. Since $\psi$ is minimal, the set 
\[\mBc{(1Bii)}=\{a_1,c_2,c_2a_2,s_4,\dots,s_n\}\]
is $2$-independent, and hence it is a $2$-basis of $N_F(\psi)$. Since $a_2\in N_F(\psi)$, the element $a_2$ has a unique expression with respect to $\mBc{({1Bii})}$: 
\begin{equation}\label{Eq:a_2InTheCase1Bii}
a_2=c_2\cdot c_2a_2\cdot \left(\frac{1}{c_2}\right)^2.
\end{equation}
It follows by Lemma~\ref{Lemma:pforms_Minimal_Repr} that $a_2\notin D_F(\psi)$.

Furthermore, as $a_1+a_2\in D_F(\varphi)\setminus F^2$, we use \eqref{Eq:MainArgue} to find $d_2\in F^*$ such that $d_2\sqf{1,a_1+a_2}\subseteq\psi$. In particular, $d_2\in D_F(\psi)$; hence,
\[d_2=x_0^2+a_1\cdot x_1^2+c_2\cdot x_2^2+c_2a_2\cdot x_3^2+\sum_{i=4}^ns_i\cdot x_i^2\]
for some suitable $x_0,\dots,x_n\in F$. Multiplying $d_2$ by $(a_1+a_2)$ and using \eqref{Eq:a_2InTheCase1Bii} (i.e., expressing $d_2(a_1+a_2)$ with respect to the basis $\mBc{({1Bii})}$), we get
\begin{align*}
d_2(a_1+a_2)&
=(a_1x_1)^2
+a_1\cdot x_0^2
+c_2\cdot (a_2x_3)^2
+c_2a_2\cdot x_2^2\\
&\quad+a_1\cdot c_2\cdot x_2^2
+a_1\cdot c_2a_2\cdot x_3^2
+\sum_{i=4}^na_1\cdot s_i\cdot x_i^2
+c_2\cdot c_2a_2\cdot \left(\frac{x_0}{c_2}\right)^2\\
&\quad+a_1\cdot c_2\cdot c_2a_2\cdot \left(\frac{x_1}{c_2}\right)^2
+\sum_{i=4}^nc_2\cdot c_2a_2\cdot s_i\cdot \left(\frac{x_i}{c_2}\right)^2.
\end{align*}
We know that $d_2(a_1+a_2)$ is represented by $\psi$; but that is impossible by Lemma~\ref{Lemma:pforms_Minimal_Repr} unless all the terms composed from at least two elements of $\mBc{({1Bii})}$ (i.e., all but the first four) are zero. It follows $x_i=0$ for all $0\leq i\leq n$; hence, $d_2=0$ which is absurd. Therefore, this case cannot happen at all.

\smallskip

We can conclude step (1): We have proved that $\psi$ has a subform similar to $\sqf{1,a_1,a_2}$. If $\dim\psi=3$, then there is nothing more to prove. From now on, we will assume that $\dim\psi\geq4$ and $\sqf{1,a_1,a_2}\subseteq\psi$.

\bigskip

(2) Now let $k\in\{3,\dots,n\}$. By \eqref{Eq:MainArgue}, we find $c_k,d_k,e_k\in F^*$ such that $c_k\sqf{1,a_k}$, $d_k\sqf{1,a_1+a_k}$ and $e_k\sqf{1,a_2+a_k}$ are subforms of $\psi$. The case distinction is slightly different than in (1); here it depends on whether $c_k$ is represented by $D_F(\sqf{1,a_1,a_2})$ or not.

\smallskip

(2A) Assume $c_k\in D_F(\sqf{1,a_1,a_2})$. Then 
\begin{equation}\label{Eq:c_kInCaseTwoA}
c_k=u_0^2+a_1\cdot u_1^2+a_2\cdot u_2^2
\end{equation}
for some $u_0,u_1,u_2\in F$ (here we slightly abuse the notation; technically, $u_0, u_1,u_2$ depend on $k$), and so
\[c_ka_k=a_k\cdot u_0^2+a_1\cdot a_k\cdot u_1^2+a_2\cdot a_k\cdot u_2^2.\]
By the uniqueness of the expression of $c_ka_k$ with respect to $\mBc{(0)}$, it follows by Lemma~\ref{Lemma:pforms_Minimal_Repr} that $c_ka_k\notin D_F(\sqf{1,a_1,a_2})$. On the other hand, we know that $c_ka_k\in D_F(\psi)$; therefore, $\psi\simeq\sqf{1,a_1,a_2,c_ka_k,s_4,\dots,s_n}$ for some $s_4,\dots, s_n\in F^*$ (possibly different from the $s_i$'s in case (1Bii) above), and 
\[\mBc{({2A})}=\{a_1,a_2,c_ka_k,s_4,\dots,s_n\}\] 
is a $2$-basis of $N_F(\psi)$ by the minimality of $\psi$. Obviously, $a_k=c_k\cdot c_ka_k\cdot c_k^{-2}$; rewriting $c_k$ via \eqref{Eq:c_kInCaseTwoA}, we get the unique expression of $a_k$ with respect to $\mBc{({2A})}$:
\begin{equation} \label{Eq:a_kInCaseTwoA}
a_k=c_ka_k\cdot \left(\frac{u_0}{c_k}\right)^2+a_1\cdot c_ka_k\cdot \left(\frac{u_1}{c_k}\right)^2+a_2\cdot c_ka_k\cdot \left(\frac{u_2}{c_k}\right)^2.
\end{equation}
Furthermore, we have $d_k\in D_F(\psi)$, and hence 
\[
d_k=x_0^2+a_1\cdot x_1^2+a_2\cdot x_2^2+c_ka_k\cdot x_3^2+\sum_{i=4}^ns_i\cdot x_i^2
\]
for suitable $x_0,\dots,x_n\in F$ (again, we omit to express the dependence on $k$). Multiplying $d_k$ by $(a_1+a_k)$ and using \eqref{Eq:a_kInCaseTwoA}, we can express $d_k(a_1+a_k)$ with respect to $\mBc{({2A})}$ as follows:
\begin{align*}
d_k(a_1+a_k)
&=(a_1x_1+a_ku_0x_3)^2
+a_1\cdot (x_0+a_ku_1x_3)^2
+a_2\cdot (a_ku_2x_3)^2 \\
&\quad+c_ka_k\cdot \left(\frac{u_0}{c_k}x_0+\frac{a_1u_1}{c_k}x_1+\frac{a_2u_2}{c_k}x_2\right)^2 
+a_1\cdot a_2\cdot x_2^2\\
&\quad+a_1\cdot c_ka_k\cdot \left(\frac{u_1}{c_k}x_0+\frac{u_0}{c_k}x_1+x_3\right)^2
+\sum_{i=4}^na_1\cdot s_i\cdot x_i^2 \\
&\quad+a_2\cdot c_ka_k\cdot \left(\frac{u_2}{c_k}x_0+\frac{u_0}{c_k}x_2\right)^2
+\sum_{i=4}^nc_ka_k\cdot s_i\cdot \left(\frac{u_0}{c_k} x_i\right)^2 \\
&\quad+a_1\cdot a_2\cdot c_ka_k\cdot \left(\frac{u_2}{c_k}x_1+\frac{u_1}{c_k}x_2\right)^2
+\sum_{i=4}^na_1\cdot c_ka_k\cdot s_i\cdot \left(\frac{u_1}{c_k} x_i\right)^2 \\
&\quad+\sum_{i=4}^na_2\cdot c_ka_k\cdot s_i\cdot \left(\frac{u_2}{c_k} x_i\right)^2 
\end{align*}
Similarly as before, since $d_k(a_1+a_k)\in D_F(\psi)$ and $\psi$ is minimal, it follows from Lemma~\ref{Lemma:pforms_Minimal_Repr} that all the \uv{composed} terms, i.e., all the terms except for the first four, must be zero. In particular:
\begin{itemize}
	\item The coefficients by $c_ka_k\cdot s_i$ and $a_1\cdot c_ka_k\cdot s_i$ and $a_2\cdot c_ka_k\cdot s_i$ are zero for each $4\leq i\leq n$; since at least one of $u_0$, $u_1$ and $u_2$ is nonzero, we get that $x_i=0$ for each $4\leq i\leq n$.
	\item The coefficient by $a_1\cdot a_2\cdot$ equals zero; thus, $x_2=0$.
	\item The coefficient by $a_2\cdot c_ka_k$ must be zero, i.e., $\frac{u_2}{c_k}x_0+\frac{u_0}{c_k}x_2=0$. As $x_2=0$, we get $u_2x_0=0$.
	\item The coefficient by $a_1\cdot a_2\cdot c_ka_k$ is zero, so $\frac{u_2}{c_k}x_1+\frac{u_1}{c_k}x_2=0$. Again, as $x_2=0$, it must hold $u_2x_1=0$.
	\item The coefficient by $a_1\cdot c_ka_k$ must be zero, thus $x_3+\frac{u_1}{c_k}x_0+\frac{u_0}{c_k}x_1=0$.
\end{itemize} 
If $u_2\neq0$, then $x_0=x_1=0$, and in that case also $x_3=0$ (by the last bullet point); but then $d_k=0$, which is absurd. Therefore, $u_2=0$.

We proceed analogously for $e_k$ and $e_k(a_2+a_k)$, only now we know $u_2=0$, which means that
\[c_k=u_0^2+a_1\cdot u_1^2 \qquad \text{ and } \qquad a_k=c_ka_k\cdot \left(\frac{u_0}{c_k}\right)^2+a_1\cdot c_ka_k\cdot \left(\frac{u_1}{c_k}\right)^2.\]
Since $e_k\in D_F(\psi)$, we have
\[e_k=y_0^2+a_1\cdot y_1^2+a_2\cdot y_2^2+c_ka_k\cdot y_3^2+\sum_{i=4}^ns_i\cdot y_i^2\]
for some $y_o,\dots,y_n\in F$, and
\begin{align*}
e_k(a_2+a_k)
=& \left(a_2y_2+c_ka_k\frac{u_0}{c_k}y_3\right)^2
+a_1\cdot \left(c_ka_k\frac{u_1}{c_k}y_3\right)^2
+a_2\cdot y_0^2\\
&+c_ka_k\cdot \left(\frac{u_0}{c_k}y_0+a_1\frac{u_1}{c_k}y_1\right)^2
 +a_1\cdot a_2\cdot y_1^2\\
&+a_1\cdot c_ka_k\cdot \left(\frac{u_1}{c_k}y_0+\frac{u_0}{c_k}y_1\right)^2
+a_2\cdot c_ka_k\cdot \left(y_3+\frac{u_0}{c_k}y_2\right)^2\\
& +\sum_{i=4}^na_2\cdot s_i\cdot y_i^2
+\sum_{i=4}^nc_ka_k\cdot s_i\cdot \left(\frac{u_0}{c_k}y_i\right)^2\\
&+a_1\cdot a_2\cdot c_ka_k\cdot \left(\frac{u_1}{c_k}y_2\right)^2
+\sum_{i=4}^na_1\cdot c_ka_k\cdot s_i\cdot \left(\frac{u_1}{c_k}y_i\right)^2.
\end{align*}
Again, all the coefficients by the composed terms must be zero, so in particular:
\begin{itemize}
	\item $y_i=0$ for all $4\leq i\leq n$ because of the coefficients by $a_2\cdot s_i$.
	\item $y_1=0$ because of the coefficient by $a_1\cdot a_2$.
	\item $u_1y_0=0$ because of the coefficient by $a_1\cdot c_ka_k$.
	\item $u_1y_2=0$ because of the coefficient by $a_1\cdot a_2 \cdot c_ka_k$.
	\item $y_3+\frac{u_0}{c_k}y_2=0$ because of the coefficient by $a_2\cdot c_ka_k$.
\end{itemize}
If $u_1\neq0$, then necessarily $y_0=0$ and $y_2=0$, which implies $y_3=0$; it would follow $e_k=0$, which is absurd. Thus, we have $u_1=0$. Therefore, $c_k=u_0^2$, and we have $\sqf{1,a_k}\subseteq\psi$; in particular, $a_k\in D_F(\psi)$.

\smallskip

(2B) Suppose $c_k\notin D_F(\sqf{1,a_1,a_2})$; then $\sqf{1,a_1,a_2,c_k}\subseteq\psi$. Here we distinguish two subcases:

\smallskip

(2Bi) Let $c_ka_k\in D_F(\sqf{1,a_1,a_2,c_k})$; then
\[c_ka_k=u_0^2+a_1\cdot u_1^2+a_2\cdot u_2^2+c_k\cdot u_3^2\]
for some $u_0,\dots,u_3\in F$, and hence
\[a_k=u_3^2+c_k\cdot \left(\frac{u_0}{c_k}\right)^2+a_1\cdot c_k\cdot \left(\frac{u_1}{c_k}\right)^2+a_2\cdot c_k\cdot \left(\frac{u_2}{c_k}\right)^2.\] 
We have 
\[\psi\simeq\sqf{1,a_1,a_2,c_k,s_4,\dots,s_n}\] 
for some $s_4,\dots,s_n\in F$; then
\[\mBc{({2Bi})}=\{a_1,a_2,c_k,s_4,\dots,s_n\}\] 
is a $2$-basis of $N_F(\psi)$. As before, we consider the unique representations of  $d_k$ and $e_k$ by $\psi$,
and express $d_k(a_1+a_k)$ and $e_k(a_2+a_k)$ with respect to $\mBc{({2Bi})}$. We obtain that $u_1=u_2=0$; therefore, $a_k=u_3^2+c_k\left(\frac{u_0}{c_k}\right)^2$, so in particular $a_k\in D_F(\psi)$.

\smallskip

(2Bii) If $c_ka_k\notin D_F(\sqf{1,a_1,a_2,c_k})$, then we have $\dim\psi\geq5$ and
\[\psi\simeq\sqf{1,a_1,a_2,c_k,c_ka_k,s_5,\dots,s_n}\] 
for some $s_5,\dots,s_n\in F$, and the corresponding $2$-basis of $N_F(\psi)$ is 
\[\mBc{({2Bii})}=\{a_1,a_2,c_k,c_ka_k,s_5,\dots,s_n\}.\] 
Again, we consider the unique representation of $d_k$ by $\psi$. In this case, we express $a_k$ with respect to $\mBc{({2Bii})}$ as
\[a_k=c_k\cdot c_ka_k\cdot \left(\frac{1}{c_k}\right)^2.\]
This time, the consideration of the element $d_k(a_1+a_k)$ with respect to $\mBc{({2Bii})}$ already implies $d_k=0$; that is absurd, and hence this case cannot happen.

\bigskip

(3) We have proved that, up to multiplying $\psi$ by a constant from $F^*$, we have $\sqf{1,a_1,a_2}\subseteq\psi$ (step (1)), and $a_k\in D_F(\psi)$ for all $3\leq k\leq n$ (step~(2)). Invoking Lemma~\ref{Lem:p-subform}, we get $\varphi\simsim\psi$.
\end{proof}

\subsection{Special quasi-Pfister neighbors}

The definition of a special quasi-Pfister neighbor is motivated by its counterpart in characteristic different from 2, which appeared in the paper \cite{AO95}.

\begin{definition}
We call a totally singular quadratic form $\varphi$ over $F$ a \emph{special quasi-Pfister neighbour} if $\varphi\simsim\pi\ort b\sigma$ with $\pi$ a quasi-Pfister form over $F$, $b\in F^*$ and $\sigma\subseteq\pi$. In such situation, we also say that $\varphi$ is given by the triple $(\pi,b,\sigma)$.
\end{definition}

For a totally singular quadratic form $\varphi$ over $F$, we define its \emph{full splitting pattern} as
\[\fsp{\varphi}=\{\dim(\varphi_E)_\an~|~E/F \text{ a field extension}\}.\]

We will describe all special quasi-Pfister neighbors of dimension up to eleven and their full splitting pattern. In particular, we will see that all quasi-Pfister neighbors up to dimension eight are special. First, we recall a proposition from \cite{HL04}.

\begin{lemma}[{\cite[Prop.~8.12]{HL04}}]\label{Lemma:PNsmalldim_TSQF}
Let $\varphi$ be an anisotropic totally singular quadratic form over $F$ with $\dim\varphi\leq8$. Then $\varphi$ is a quasi-Pfister neighbor if and only if 
\begin{enumerate}
		\item $\dim\varphi\leq3$, or 
		\item $\dim\varphi=2^n$ for some $n\geq1$ and $\varphi$ is similar to an $n$-fold quasi-Pfister form, or
		\item there exist $x,y,z\in F^*$ such that
			\begin{enumerate}
				\item $\varphi\simsim\sqf{1,x,y,xy,z}$ in the case of $\dim\varphi=5$,
				\item $\varphi\simsim\sqf{1,x,y,xy,z,xz}$ in the case of $\dim\varphi=6$,
				\item $\varphi\simsim\sqf{x,y,z,xy,xz,yz,xyz}$ in the case of $\dim\varphi=7$.
			\end{enumerate}
\end{enumerate}
\end{lemma}

With the previous lemma, it is easy to classify all special quasi-Pfister neighbors of dimensions up to $8$: Any quasi-Pfister neighbor of dimension $2^n$ for some $n\geq0$ is similar to an $n$-fold quasi-Pfister form, and hence special. For the other small dimensions, we have
\begin{itemize}
	\item $\sqf{1,x,y}\simeq\pf{x}\ort y\sqf{1}$,
	\item $\sqf{1,x,y,xy,z}\simeq\pf{x,y}\ort z\sqf{1}$,
	\item $\sqf{1,x,y,xy,z,xz}\simeq\pf{x,y}\ort z\sqf{1,x}$,
	\item $\sqf{x,y,z,xy,xz,yz,xyz}\simeq xyz(\pf{x,y}\ort z\sqf{1,x,y})$;
\end{itemize}
therefore, we get the following corollary.

\begin{corollary} \label{Cor:PNsmalldimSPN_TSQF}
All anisotropic quasi-Pfister neighbors of dimensions up to $8$ are special.
\end{corollary}

\begin{proposition} \label{Prop:SPNsmalldim_TSQF}
Let $\varphi$ be an anisotropic special quasi-Pfister neighbor.
\begin{enumerate}
	\item If $\dim\varphi=3$, then we have $\varphi\simsim\pf{a}\ort d\sqf{1}$ for some $a,d\in F^*$ and ${\fullsplitpat{\varphi}=\{1,2,3\}}$.
	\item If $\dim\varphi=5$, then  $\varphi\simsim\pf{a,b}\ort d\sqf{1}$ for some $a,b, d\in F^*$ and $\fullsplitpat{\varphi}=\{1,2,3,4,5\}$.
	\item If $\dim\varphi=6$, then  $\varphi\simsim\pf{a,b}\ort d\sqf{1,a}$ for some $a,b, d\in F^*$ and $\fullsplitpat{\varphi}=\{1,2,3,4,6\}$.
	\item If $\dim\varphi=7$, then $\varphi\simsim\pf{a,b}\ort d\sqf{1,a,b}$ for some $a,b, d\in F^*$ and $\fullsplitpat{\varphi}=\{1,2,4,7\}$.
	\item If $\dim\varphi=9$, then $\varphi\simsim\pf{a,b,c}\ort d\sqf{1}$ for some $a,b,c, d\in F^*$ and $\fullsplitpat{\varphi}=\{1,2,3,4,5,8,9\}$.
	\item If $\dim\varphi=10$, then  $\varphi\simsim\pf{a,b,c}\ort d\sqf{1,a}$ for some ${a,b,c, d\in F^*}$ and $\fullsplitpat{\varphi}=\{1,2,3,4,5,6,8,10\}$.
	\item If $\dim\varphi=11$, then $\varphi\simsim\pf{a,b,c}\ort d\sqf{1,a,b}$ for some $a,b,c, d\in F^*$ and $\fullsplitpat{\varphi}=\{1,2,3,4,6,7,8,11\}$.
\end{enumerate}
\end{proposition}

\begin{proof}
All special quasi-Pfister neighbors up to dimension $8$ have been described by Lemma~\ref{Lemma:PNsmalldim_TSQF} and Corollary~\ref{Cor:PNsmalldimSPN_TSQF}. If $\varphi\simsim\pi\ort d\sigma$ is a special quasi-Pfister neighbor and $9\leq\dim\varphi\leq11$, then necessarily $\dim\pi=2^3$ and $1\leq\dim\sigma\leq3$. Recall that all totally singular quadratic forms of dimensions two or three are minimal; thus, without loss of generality, $\sqf{1}\subseteq\sigma\subseteq\sqf{1,a,b}$ for some $2$-independent set $\{a,b\}\subseteq F$. Then we can find $c\in F^*$ such that $\{a,b,c\}$ is a $2$-basis of $N_F(\pi)$ over $F$, and it follows that $\pi\simeq\pf{a,b,c}$.

The full splitting pattern of $\varphi$ follows directly from \cite[Thm.~4.11]{KZ23-fsp}.
\end{proof}

In the following example, we will show that not all quasi-Pfister neighbors are special.

\begin{example}
Let $\varphi\simeq\sqf{1,a,b,c,d,ab,ac,ad,bc}$ be a totally singular quadratic form such that $\{a,b,c,d\}\subseteq F$ is $2$-independent over $F$. Then $\varphi$ is a quasi-Pfister neighbor of $\pf{a,b,c,d}$.

On the other hand, we have
\[\varphi_{F(\sqrt{b})}\simeq\sqf{1,a,1,c,d,a,ac,ad,c}_{F(\sqrt{b})}\simeq\sqf{1,a,c,d,ac,ad,0,0,0}_{F(\sqrt{b})};\]
since $\sqf{1,a,c,d,ac,ad} \subseteq\pf{a,c,d}$ and $(\pf{a,c,d})_{F(\sqrt{b})}\simeq(\pf{a,b,c,d}_{F(\sqrt{b})})_\an$ is anisotropic, we get 
\[(\varphi_{F(\sqrt{b})})_{\an}\simeq\sqf{1,a,c,d,ac,ad}_{F(\sqrt{b})},\]
so we have in particular $6\in\fullsplitpat{\varphi}$. As the full splitting pattern of any $9$-dimensional special quasi-Pfister neighbor equals to $\{1,2,3,4,5,8,9\}$ by Proposition~\ref{Prop:SPNsmalldim_TSQF}, it follows that $\varphi$ cannot be any special quasi-Pfister neighbor.
\end{example}

\bigskip

Ideally, we would like to prove that if two totally singular forms are Vishik equivalent and at least one of them is a special quasi-Pfister neighbor, then they are similar. Unfortunately, we will need some additional assumptions. We start with a few lemmas.

\begin{lemma}[{\cite[Lemma~4.8]{KZ23-fsp}}]\label{Lemma:IsotropyIndicesSPN_pforms}
Let $\pi$ be a quasi-Pfister form over $F$, $\sigma\subseteq\pi$ and $d\in F^*$ be such that the totally singular quadratic form $\varphi\simeq\pi\ort d\sigma$ is anisotropic. Let $E/F$ be a field extension.
\begin{enumerate}
	\item If $d\in D_E(\pi)$, then we have $(\varphi_E)_{\an}\simeq(\pi_E)_{\an}$; in particular, $\iql{\varphi_E}=\iql{\pi_E}+\dim\sigma$.
	\item If $d\notin D_E(\pi)$, then $(\varphi_E)_{\an}\simeq(\pi_E)_{\an}\ort d(\sigma_E)_{\an}$; in particular, $\iql{\varphi_E}=\iql{\pi_E}+\iql{\sigma_E}$.
\end{enumerate}
\end{lemma}

We prove that if a quasi-Pfister neighbor of norm degree $2^n$ contains a quasi-Pfister form of dimension $2^{n-1}$, then it must be a special quasi-Pfister neighbor.

\begin{lemma} \label{Lemma:TSQF_SPNwithknownPF}
Let $\pi$ be an anisotropic quasi-Pfister form over $F$. Moreover, let $\psi$ be an anisotropic totally singular quadratic form over $F$ such that $\pi\subseteq\psi$ and $\nf_F(\psi)\simeq\pi\otimes\pf{b}$ for some $b\in F^*$. Then there exists a totally singular form $\rho\subseteq\pi$ such that $\psi\simeq\pi\ort b\rho$.
\end{lemma}

\begin{proof}
First, recall that a norm form is always anisotropic; hence, $\pi\otimes\pf{b}$ is anisotropic, and we get by Lemma~\ref{Lem:quasiPFaddingSlot} that $b\notin N_F(\pi)=D_F(\pi)$.

Let $\rho'$ be a totally singular quadratic form over $F$ such that $\psi\simeq \pi\ort b\rho'$. Denote $s=\dim\rho'$ and $\rho'=\sqf{d'_1,\dots,d'_s}$. For each $k\in\{1,\dots,s\}$, we proceed as follows: Since $bD_F(\rho')\subseteq D_F(\psi)\subseteq D_F(\pi\otimes\pf{b})$, we can write 
\[bd'_k=\pi(\boldxi_k)+b\pi(\boldzeta_k)\] 
for some appropriate vectors $\boldxi_k,\boldzeta_k$. If $\boldzeta_k=0$, then $bd'_k=\pi(\boldxi_k)\in D_F(\pi)$, which contradicts the anisotropy of $\psi$; thus, the vector $\boldzeta_k$ must be nonzero. It follows that $\pi(\boldzeta_k)\neq0$, and so we set $d_k=\pi(\boldzeta_k)$; then
\[\pi\ort b\sqf{d'_k}\simeq\pi\ort\sqf{\pi(\boldxi_k)+bd_k}\simeq\pi\ort b\sqf{d_k}.\]
It follows that
\[\pi\ort b\sqf{d'_1,\dots,d'_k}\ort b\sqf{d_{k+1},\dots, d_s}\simeq \pi\ort b\sqf{d'_1,\dots,d'_{k-1}}\ort b\sqf{d_{k},\dots, d_s}\]
for any $1\leq k\leq s$. Therefore,
\[\pi\ort b\sqf{d'_1,\dots,d'_s}\simeq\pi\ort b\sqf{d_1,\dots, d_s},\]
where $\rho\simeq\sqf{d_1,\dots, d_s}$ is a subform of $\pi$, because $d_k\in D_F(\pi)$ for each $k$.
\end{proof}

\begin{proposition} \label{Prop:TSQF_SPN}
Let $\varphi,\psi$ be anisotropic totally singular quadratic forms over $F$ such that $\varphi\simvw\psi$.  Assume that $\varphi$ is a special quasi-Pfister neighbor given by a triple $(\pi, b,\sigma)$. Moreover, suppose that $c\pi\subseteq\psi$ for some $c\in F^*$. Then $\psi$ is a special quasi-Pfister neighbor given by a triple $(\pi,b,\rho)$ for some form $\rho$ over $F$ such that $\rho\simvw\sigma$.
\end{proposition}

\begin{proof}
Let $d\in D_F(\sigma)$. Then $d\in D_F(\pi)=G_F(\pi)$; hence, $d\varphi\simeq\pi\ort b(d\sigma)$ and $1\in D_F(d\sigma)$. Therefore, we can assume that $1\in D_F(\sigma)$. Then, since $\varphi$ is anisotropic, we must have $b\notin D_F(\pi)$. It follows that $\nf_F(\varphi)\simeq\pi\otimes\pf{b}$. Since $\nf_F(\psi)\simeq\nf_F(\varphi)$ by Proposition~\ref{Prop:WeakVishikSameNormField}, we also have $\nf_F(\psi)\simeq\pi\otimes\pf{b}$. Finally, we can suppose that $\pi\subseteq\psi$. Therefore, we can apply Lemma~\ref{Lemma:TSQF_SPNwithknownPF} to find a totally singular quadratic form $\rho\subseteq\pi$ such that $\psi\simeq\pi\ort b\rho$. 

It remains to show that $\sigma\simvw\rho$: The equality $\dim\sigma=\dim\rho$ follows directly from $\dim\varphi=\dim\psi$. So, consider the field $E=F(\sqrt{a})$ for some $a\in F^*$. If $a\notin N_F(\pi)$, then $\pi_E$ is anisotropic by Lemma~\ref{Lemma:p-forms_BasicIsotropy}, and so are its subforms $\sigma_E$ and $\rho_E$; in particular, $\iql{\sigma_E}=\iql{\rho_E}$. On the other hand, if $a\in N_F(\pi)$, then $(\pi\otimes\pf{a})_\an\simeq\pi$ by Lemma~\ref{Lem:quasiPFaddingSlot}; together with Lemma~\ref{Lemma:p-forms_BasicIsotropy}, we obtain
\[D_E(\pi_E)=D_F(\pi\otimes\pf{a})=D_F(\pi).\]
It follows that $b\notin D_E(\pi_E)$, and so we get by Lemma~\ref{Lemma:IsotropyIndicesSPN_pforms} that
\[\iql{\varphi_E}=\iql{\pi_E}+\iql{\sigma_E}\qquad \text{and} \qquad \iql{\psi_E}=\iql{\pi_E}+\iql{\rho_E}.\]
Since $\iql{\varphi_E}=\iql{\psi_E}$ by the assumption, we get $\iql{\sigma_E}=\iql{\rho_E}$. 
\end{proof}

\begin{remark} 
With the same notation as in Proposition~\ref{Prop:TSQF_SPN}, we can consider a stronger assumption $\varphi\simv\psi$ and ask whether it implies $\sigma\simv\rho$. 

First, note that any field $E$ with $b\in D_E(\pi)$ is problematic: In this case, we have 
\[b\varphi_E\simeq b\pi_E\ort\sigma_E\simeq (\pi\ort\sigma)_E.\]
Hence, $(\varphi_E)_{\an}\simeq(\pi_E)_{\an}$, so we do not get any information about $\iql{\sigma_E}$.

We can still give a more specific characterization of the problematic fields: As in the proof of Proposition~\ref{Prop:TSQF_SPN}, we can assume that $1\in D_F(\sigma)$. First, let $T$ and $S$ be fields such that $F\subseteq T\subseteq S\subseteq E$, where $T/F$ is purely transcendental, $S/T$ is separable and $E/S$ is purely inseparable. 

If $b\in D_S(\pi)$, then $(\pi\ort\sqf{b})_S$ is isotropic, and hence $\pi\ort\sqf{b}$ must be isotropic over $F$. But that is impossible because $\pi\ort\sqf{b}\subseteq\pi\ort b\sigma$ and $\pi\ort b\sigma$ is anisotropic. Thus, $b\notin D_S(\pi)$.

Furthermore, isotropy is a finite problem; thus, we can construct a field $L$ with $S\subseteq L\subseteq E$ such that $L/S$ is finite, and we have $\iql{\sigma_L}=\iql{\sigma_E}$ and $\iql{\rho_L}=\iql{\rho_E}$. Then $L=S\bigl(\!\!\!\sqrt[\leftroot{-5}\uproot{4}2^{n_1}]{s_1},\dots,\!\!\!\sqrt[\leftroot{-5}\uproot{4}2^{n_k}]{s_k}\bigr)$ for some $k\geq0$, $s_i\in S$ and $n_i\geq1$. Set $K=S(\sqrt{s_1},\dots,\sqrt{s_k})$; then $\iql{\sigma_K}=\iql{\sigma_L}$ and $\iql{\rho_K}=\iql{\rho_L}$ by \cite[Thm.~3.3]{KZ23-fsp}. 

If $b\notin D_K(\pi)$, then we can apply Lemma~\ref{Lemma:IsotropyIndicesSPN_pforms} to get $\iql{\sigma_K}=\iql{\rho_K}$; then $\iql{\sigma_E}=\iql{\rho_E}$ by the construction of $K$, and we are done. Therefore, assume $b\in D_K(\pi)$.

Now we can construct a field $M$ such that $S\subseteq M\subseteq K$, it holds that $b\notin D_M(\pi)$, and for each $\ve\in K\setminus M$, we have $b\in D_{M(\ve)}(\pi)$. Then we have $(\varphi_M)_{\an}\simeq(\pi_M)_{\an}\ort b (\sigma_M)_{\an}$ and $(\psi_M)_{\an}\simeq(\pi_M)_{\an}\ort b (\rho_M)_{\an}$ and $\iql{\sigma_M}=\iql{\rho_M}$ by Lemma~\ref{Lemma:IsotropyIndicesSPN_pforms}. Thus, let $\varphi'$, $\psi'$, $\pi'$, $\sigma'$, $\rho'$ be forms over $F$ such that $\varphi'_M\simeq(\varphi_M)_\an$, $\psi'_M\simeq(\psi_M)_\an$, etc. In particular, $\pi'_M$ is a quasi-Pfister form by Lemma~\ref{Lemma:PFandPN}.

Let $\ve\in K\setminus M$, and $e\in S$ be such that $\ve^2=e$. Then we know 
\[b\in D_{M(\sqrt{e})}(\pi')\setminus D_M(\pi')=D_M(\pi'\otimes\pf{e})\setminus D_M(\pi');\]
in particular, $D_M(\pi')\subsetneq D_M(\pi'\otimes\pf{e})$, which is only possible if $(\pi'\otimes\pf{e})_M$ is anisotropic. It follows by Lemma~\ref{Lemma:p-forms_BasicIsotropy} that $\pi'_{M(\sqrt{e})}$ is anisotropic. In that case, its subforms $\sigma'_{M(\sqrt{e})}$, $\rho'_{M(\sqrt{e})}$ must be anisotropic, too. It follows that
\[\iql{\sigma_{M(\sqrt{e})}}=\iql{\sigma_M}=\iql{\rho_M}=\iql{\rho_{M(\sqrt{e})}}.\]

\smallskip

However, the field extension $K/M$ does not have to be simple, as we show in the following example: Let $\pi'_M\simeq\pf{a_1,\dots,a_n}_M$ for some $a_1,\dots,a_n\in F^*$, and assume that $n\geq2$. Set $K=M(\sqrt{b}, \sqrt{a_1+a_2b})$. It is easy to see that $\{b, a_1+a_2b\}$ is $2$-independent over $M$, and hence $[K:M]=4$. Since $K$ does not contain any element of degree four over $M$, it follows that $K/M$ is not simple. Now consider $\ve\in K\setminus M$; then
\[\ve=w+x\sqrt{b}+y\sqrt{b}\sqrt{a_1+a_2b}+z\sqrt{a_1+a_2b}\]
with $w,x,y,z\in M$, at least one of $x, y,z$ nonzero. Then
\[
D_{M(\ve)}(\pi')
= D_M(\pi'\otimes\pf{\ve^2})
= M^2(a_1,\dots,a_n,\ve^2)
\]
Considering $\ve^2$ as an element of $K^2/M^2(a_1,\dots,a_n)$, we get:
\begin{multline*}
\ve^2
=w^2+bx^2+a_1by^2+a_2(by)^2+a_1z^2+a_2bz^2
\equiv bx^2+a_1by^2+a_2bz^2 \\
= b(x^2+a_1y^2+a_2z^2)
\equiv b \mod M^2(a_1,\dots,a_n),
\end{multline*}
where we used that $x^2+a_1y^2+a_2z^2\neq0$ (this holds because $\{a_1,a_2\}$ is $2$-independent over $M$ and at least one of $x,y,z$ is nonzero by the assumption). Therefore, we have $D_{M(\ve)}(\pi')= M^2(a_1,\dots,a_n, b)$. In particular, $b\in D_{M(\ve)}(\pi')$ for any $\ve\in K\setminus M$, and so we cannot find any field $M'$ with $M\subsetneq M'\subseteq K$ such that $b\notin D_{M'}(\pi')$.
\end{remark}

With the notation and assumptions of Proposition~\ref{Prop:TSQF_SPN}, we know that $\sigma\simvw\rho$, and hence $N_F(\sigma)\simeq N_F(\rho)$ by Proposition~\ref{Prop:WeakVishikSameNormField}. To conclude this section, we prove $N_F(\sigma)\simeq N_F(\rho)$ by a different approach.

\begin{lemma} \label{Lemma:TSQF_SPNnormfield}
Let $\varphi=\pi\ort b\sigma$, $\psi=\pi\ort b \rho$ be anisotropic totally singular quadratic forms over $F$ with $\pi$ a quasi-Pfister form, $b\in F^*$ and $\sigma,\rho\subseteq\pi$. Moreover, suppose that $1\in D_F(\sigma)\cap D_F(\rho)$. If $\varphi\simv\psi$, then $N_F(\sigma)=N_F(\rho)$.
\end{lemma}

\begin{proof}
First, note that the anisotropy of $\varphi$ together with the assumption $1\in D_F(\sigma)$ implies that $b\notin D_F(\pi)$.

By Lemma~\ref{Lem:BasisSubform}, we can find $c_1,\dots,c_{s}\in F^*$ and $0\leq k\leq s$ such that $\sigma=\sqf{1,c_1,\dots,c_{s}}$ and $\{c_1,\dots,c_k\}$ is a $2$-basis of $N_F(\sigma)$ over $F$. Set $K=F(\sqrt{c_1},\dots,\sqrt{c_k})$. Then $(\sigma_K)_{\an}\simeq\sqf{1}$ by Proposition~\ref{Prop:p-forms_NormfieldCharacterisation}. Write $\pi\simeq\pf{a_1,\dots,a_n}$ and $\pi'\simeq(\pi_K)_{\an}$. Then
\begin{multline*}
D_K(\pi')=K^2(a_1,\dots,a_n)=F^2(c_1,\dots,c_k,a_1,\dots,a_n)\\
\stackrel{N_F(\sigma)\subseteq N_F(\pi)}{=}F^2(a_1,\dots,a_n)=D_F(\pi);
\end{multline*}
hence, $b\notin D_K(\pi')$, and so $(\varphi_K)_{\an}\simeq\pi'\ort b\sqf{1}$. We will show that this is isometric to $(\psi_K)_{\an}$: Obviously, $\pi'\subseteq (\psi_K)_{\an}$. From the Vishik equivalence of $\varphi$ and $\psi$ we know $\dim(\psi_K)_{\an}=\dim\pi'+1$. Set $\rho'\simeq(\rho_K)_{\an}$. Since $b\notin D_K(\pi')$, we get by Lemma~\ref{Lemma:IsotropyIndicesSPN_pforms} that $\pi'\ort b\rho'$ is anisotropic. It means that $\pi'\ort b\rho'\simeq(\psi_K)_{\an}$, and hence $\dim\rho'=1$. As $1\in D_F(\rho)\subseteq D_K(\rho')$, it follows that $\rho'\simeq\sqf{1}$. Consequently, again by Proposition~\ref{Prop:p-forms_NormfieldCharacterisation}, we get $N_F(\rho)\subseteq F^2(c_1,\dots,c_k)=N_F(\sigma)$. The other inclusion can be proved analogously. Therefore, $N_F(\sigma)=N_F(\rho)$.
\end{proof}

\subsection{Conclusion}

In this final part, we put together all our results on the Vishik equivalence of totally singular quadratic forms.

Before stating the main theorem, note that any anisotropic totally singular quadratic form of dimension less or equal to four is either minimal or it is similar to a quasi-Pfister form.

\begin{theorem} \label{Th:TSQF_SPNsummary}
Let $\varphi$, $\psi$ be totally singular quadratic forms over $F$ such that $\varphi\simvw\psi$. Assume that $\varphi$ is a special quasi-Pfister neighbor given by the triple $(\pi,b,\sigma)$, and $c\pi\subseteq\psi$ for some $c\in F^*$. Moreover, suppose that $\sigma$ is either a quasi-Pfister form, a quasi-Pfister neighbor of codimension one, or a minimal form. Then $\varphi\simsim\psi$.
\end{theorem}

\begin{proof}
By Proposition~\ref{Prop:TSQF_SPN}, there exists a totally singular quadratic form ${\rho\subseteq\pi}$ and $c'\in F^*$ such that $c'\psi\simeq\pi\ort b\rho$ and $\sigma\simvw\rho$. Then, by Corollary~\ref{Cor:WeakVishikAndMinimality}, resp. by Theorem~\ref{Th:TSQF_Minimal}, we have $\sigma\simsim\rho$. 

Let $d\in F^*$ be such that $\rho\simeq d\sigma$, and let $a\in D_F^*(\sigma)$. Then 
\[da\in D_F^*(d\sigma)=D_F^*(\rho)\subseteq D_F^*(\pi)=G_F^*(\pi).\]
As $a\in G_F^*(\pi)$ and $G_F^*(\pi)$ is a group, it follows that $d\in G_F^*(\pi)$. Hence,
\[d\varphi\simeq d\pi\ort b(d\sigma)\simeq\pi\ort b\rho\simeq c'\psi,\]
i.e., $\varphi\simsim\psi$.
\end{proof}

\begin{corollary}\label{Cor:VishikTSQFtotal}
Let $\varphi$, $\psi$ be totally singular quadratic forms over $F$ such that $\varphi\simv\psi$. Let $K/F$ be a field extension such that $K^2\simeq G_F(\varphi)$. Assume that $(\varphi_K)_\an$ is a special quasi-Pfister neighbor given by the triple $(\pi,b,\sigma)$, and that $c\pi\subseteq\psi_K$ for some $c\in K^*$. Moreover, suppose that $\sigma$ is either a quasi-Pfister form, a quasi-Pfister neighbor of codimension one, or a minimal form (over $K$). Then $\varphi\simsim\psi$.
\end{corollary}

\begin{proof}
Without loss of generality, assume that $\varphi$ and $\psi$ are anisotropic. Denote ${\tau\simeq\simf_F(\varphi)}$, and let $a_1,\dots,a_n\in F^*$ be such that $\tau\simeq\pf{a_1,\dots,a_n}$. Then there exists a totally singular quadratic form $\varphi'$ over $F$ such that ${\varphi\simeq\tau\otimes\varphi'}$. Furthermore, as $G_F(\varphi)=G_F(\tau)=D_F(\tau)=F^2(a_1,\dots,a_n)$, we have ${K= F(\sqrt{a_1},\dots,\sqrt{a_n})}$.  By \cite[Thm.~3.3]{KZ23-fsp}, $\iql{\varphi'_K}=\iql{\tau\otimes\varphi'}=0$.  Therefore, $\varphi'_K$ is anisotropic, and we have $(\varphi_K)_\an\simeq\varphi'_K$.  Since $G_F(\varphi)\simeq G_F(\psi)$ by Theorem~\ref{Thm:SimilarityFactorsVishik_pforms}, we can use analogous arguments to find a form $\psi'$ over $F$ such that $\psi\simeq\tau\otimes\psi'$ and $\psi'_K\simeq(\psi_K)_\an$.

Since $\varphi\simv\psi$, we get by Lemma~\ref{Lem:VishikWrtExt} that $\varphi'_K\simv\psi'_K$. Since $c\pi\subseteq\psi_K$ and $\pi$ is anisotropic over $K$, it follows that $c\pi\subseteq\psi'_K$. Note that $\varphi'_K\simeq(\varphi_K)_\an$ is a special quasi-Pfister neighbor given by the triple $(\pi,b,\sigma)$. Hence, by Theorem~\ref{Th:TSQF_SPNsummary}, we have $\varphi'_K\simsim\psi'_K$. Finally, Proposition~\ref{Prop:DividingByPF_pforms} implies that $\varphi\simsim\psi$.
\end{proof}

\begin{remark}
Note that we need in the proof of Corollary~\ref{Cor:VishikTSQFtotal} that ${\varphi'_K\simvw\psi'_K}$. To be able to conclude that, we need the full strength of $\varphi\simv\psi$, it would not be sufficient to assume $\varphi\simvw\psi$.
\end{remark}


\bigskip

We would like to conclude the article with a determination of the smallest dimension in which the answer to Question~\ref{MQ} is not fully known. To do that, we first need to characterize totally singular quadratic forms of low dimension.

\begin{proposition} \label{Prop:LowDim}
Let $\varphi$ be a totally singular quadratic form over $F$. In each of the following cases, there exists a $2$-independent set $\{a,b,c,d,e\}$ over $F$ such that:
\begin{enumerate}
	\item\label{Prop:LowDim1} Let $\dim\varphi=1$ and $\ndeg_F\varphi=1$, then $\varphi\simsim\sqf{1}$.
	\item\label{Prop:LowDim2} Let $\dim\varphi=2$ and $\ndeg_F\varphi=2$, then $\varphi\simsim\sqf{1,a}$.
	\item\label{Prop:LowDim3} Let $\dim\varphi=3$ and $\ndeg_F\varphi=4$, then $\varphi\simsim\sqf{1,a,b}$. 
	\item\label{Prop:LowDim4} Let $\dim\varphi=4$;
			\begin{itemize}
				\item if $\ndeg_F\varphi=4$, then $\varphi\simsim\pf{a,b}$,
				\item if $\ndeg_F\varphi=8$, then $\varphi\simsim\sqf{1,a,b,c}$.
			\end{itemize}
	\item\label{Prop:LowDim5} Let $\dim\varphi=5$;
			\begin{itemize}
				\item if $\ndeg_F\varphi=8$, then $\varphi\simsim\pf{a,b}\ort c\sqf{1}$,
				\item if $\ndeg_F\varphi=16$, then $\varphi\simsim\sqf{1,a,b,c,d}$.
			\end{itemize}
	\item\label{Prop:LowDim6} Let $\dim\varphi=6$;
			\begin{itemize}
				\item if $\ndeg_F\varphi=8$, then $\varphi\simsim\pf{a,b}\ort c\sqf{1,a}$,
				\item if $\ndeg_F\varphi=16$, then there are two possibilities: either $\varphi\simsim\pf{a,b}\ort\sqf{c,d}$, or $\varphi\simsim\sqf{1,a,b,c,d, t+ad}$ for some $t\in F^2(a,b,c)$,
				\item if $\ndeg_F\varphi=32$, then $\varphi\simsim\sqf{1,a,b,c,d,e}$.
			\end{itemize}
\end{enumerate}
\end{proposition}

\begin{proof}
Cases \ref{Prop:LowDim1}--\ref{Prop:LowDim4} are obvious. 

\ref{Prop:LowDim5} Let $\dim\varphi=5$. If $\ndeg_F\varphi=8$, then $\varphi$ is a quasi-Pfister neighbor, which is special by Corollary~\ref{Cor:PNsmalldimSPN_TSQF}; the claim then follows from Proposition~\ref{Prop:SPNsmalldim_TSQF}. If $\ndeg_F\varphi=16=2^{\dim\varphi-1}$, then $\varphi$ is minimal, and the claim follows from Lemma~\ref{Lem:PropertiesOfMinimalForms_pforms}.

\ref{Prop:LowDim6} Assume that $\dim\varphi=6$. If $\ndeg_F\varphi=8$, then $\varphi$ is a quasi-Pfister neighbor, special by Corollary~\ref{Cor:PNsmalldimSPN_TSQF}, and we obtain the required form from Proposition~\ref{Prop:SPNsmalldim_TSQF}. 

If $\dim\varphi=6$ and $\ndeg_F\varphi=16$, then consider $\tau\subseteq\varphi$ with $\dim\tau=5$, and let $x\in F$ be such that $\varphi\simeq\tau\ort\sqf{x}$. By \ref{Prop:LowDim5}, we can assume that either $\tau\simeq\pf{a,b}\ort\sqf{c}$, or $\tau\simeq\sqf{1,a,b,c,d}$. In the former case, we must have $x\in F^2(a,b,c,d)\setminus F^2(a,b,c)$, as otherwise $\ndeg_F\varphi=8$. Therefore, $N_F(\varphi)=F^2(a,b,c,d)=F^2(a,b,c,x)$, so we can exchange the $d$ in the $2$-basis of $N_F(\varphi)$ for $x$. Up to renaming, we get that $\varphi\simeq\pf{a,b}\ort\sqf{c,d}$. 

Now, assume $\tau\simeq\sqf{1,a,b,c,d}$, i.e., $\varphi\simeq\sqf{1,a,b,c,d,x}$ for some $x\in F^2(a,b,c,d)$. Write $x=y+dz$ with $y,z\in F^2(a,b,c)$. If $z\in F^2$, then $\varphi\simeq\sqf{1,a,b,c,d,y}$. Denote $\sigma\simeq\sqf{1,a,b,c,y}$; then $\dim\sigma=5$ and $\ndeg_F\sigma=8$, so (up to renaming) $\sigma\simeq\pf{a,b}\ort\sqf{c}$ by \ref{Prop:LowDim5}, and hence $\varphi\simeq\pf{a,b}\ort\sqf{c,d}$. Now, assume that $z\notin F^2$.  Then we can write $z=z_1+az_2$ with $z_1,z_2\in F^2(b,c)$; we can assume that $z_2\neq0$ (otherwise exchange $a$ with $b$ or $c$). Hence, $F^2(a,b,c)=F^2(z,b,c)$; without loss of generality, we can assume $a=z$. We get $\varphi\simeq\sqf{1,a,b,c,d,y+ad}$ with $y\in F^2(a,b,c)$ as claimed.

In the remaining case, $\dim\varphi=6$ and $\ndeg_F\varphi=32$, we have $\ndeg_F\varphi=2^{\dim\varphi-1}$. Hence, $\varphi$ is minimal, and we conclude by applying Lemma~\ref{Lem:PropertiesOfMinimalForms_pforms}.
\end{proof}

\begin{corollary}
The answer to Question~\ref{MQ} is positive for all totally singular quadratic forms of dimension $\leq5$.
\end{corollary}

\begin{proof}
Let $\varphi$ be a totally singular form of dimension $\leq5$. By Proposition~\ref{Prop:LowDim}, $\varphi$ is minimal, or a quasi-Pfister form, or a special quasi-Pfister neighbor given by a triple $(\pi,b,\sigma)$ with $\sigma$ minimal. Hence, the claim follows by Theorem~\ref{Th:TSQF_Minimal} (minimal forms),  Proposition~\ref{Prop:Vishik_PF} (quasi-Pfister forms), and Theorem~\ref{Th:TSQF_SPNsummary} (special quasi-Pfister neighbors).
\end{proof}

\begin{remark}
In dimension six, the only case that remains open is if the norm degree of the form is $16$. This case will be covered in a forthcoming article.
\end{remark}

\printbibliography

\end{document}